\documentclass{article}

\usepackage{amssymb}

\usepackage{amsmath,amssymb,amsfonts,amsthm,amsxtra }

\usepackage[applemac]{inputenc}

\usepackage[english]{babel}
\usepackage[T1]{fontenc}
\usepackage[all]{xy}

\usepackage{enumitem} 

\usepackage{amssymb}
\usepackage{amsmath}
\usepackage{amsthm}
\theoremstyle{plain}
\newtheorem{thm}{Theorem}[section]
\newtheorem{prop}[thm]{Proposition}

\newtheorem{nota}[thm]{Notation}
\newtheorem{lem}[thm]{Lemma}
\newtheorem{coro}[thm]{Corollary}
\newtheorem{defn}[thm]{Definition}

\newtheorem{rmq}[thm]{Remark} 
\numberwithin{equation}{section}  
\newcommand{\prin}{\operatorname{prin}}

\newcommand{\Stieffel}{\operatorname{Stiefel}}

\newcommand{\V}{\operatorname{V}}
\newtheorem{theorem}{Theorem}

\title{\textbf{Rational homotopy of the space of immersions between manifolds}}

\author{ Abdoulkader Yacouba Barma}
\begin{document}
\maketitle
\begin{abstract}
In this paper we study the rational  homotopy of the space of immersions, $Imm\left(M,N\right)$,  of a manifold $M$ of dimension $m\geq 0$ into a manifold  $N$ of dimension $m+k$, with $k\geq 2$. In the special case  when $N=\mathbb{R}^{m+k}$ and $k$ is odd we prove that each connected component of $Imm\left(M,\mathbb{R}^{m+k}\right)$ has the rational homotopy type of product of Eilenberg Mac Lane space.  We give an explicit  description of each connected component and prove that it only depends on $m$, $k$ and the rational Betti numbers of $M$. \\
For a more general manifold $N$, we prove that the path connected of $Imm\left(M,N\right)$ has the rational homotopy type of some component of an explicit mapping space when some Pontryagin classes vanishes.
\end{abstract}

\section{Introduction}
Throughout the paper $M$ and $N$ are  differentiable manifolds of dimension $m$ and $m+k$ respectively, $m,k\geq 0$. We assume that $M$ and $N$  are simply connected and of finite type. An immersion from $M$ to $N$ is a map $f\colon M\longrightarrow N$ of class $C^1$ such that for all $x$ in $M$ the derivative, $d_xf$ is injective. It will be denoted $f\colon M\looparrowright N$. This paper is about the study of the \emph{space of immersions} $Imm\left(M,N\right)$ defined as the set of all immersions  $f\colon M\looparrowright N$ equipped with the weak $C^1$-topology (see \cite[Chapter~2]{hirsh1}).
In \cite{hirsh, smale2}, M. Hirsch and S. Smale construct a homotopical description of $Imm\left(M,N\right)$ as the space of sections of some explicit bundle (see Theorem~\ref{sht} below). In this paper we study the rational homotopy type of $Imm\left(M,N\right)$. 
  Our main result is the following.
\begin{theorem}[Theorem~\ref{casodd}]
\label{theoremB}
Let $M$ be a simply connected manifold of finite type of dimension $m\geq 0$. If $k\geq 3$ is an  odd integer,
then each connected component of $Imm\left(M,\mathbb{R}^{m+k}\right)$ has the rational homotopy type of a product of Eilenberg-Mac Lane spaces which only depends on $m$, $k$ and the rational Betti numbers of $M$.

\end{theorem}
This theorem is quite surprising. Indeed Smale-Hirsch describe the space $Imm\left(M,\mathbb{R}^{m+k}\right)$ as the space of section of some bundle depending on the tangent bundle of $M$. The rational information on this tangent bundle is encoded by the Pontryagin classes and therefore one could except that the rational homotopy of  $Imm\left(M,\mathbb{R}^{m+k}\right)$ depends on those Pontryagin classes. This is not the case, the reason being that \emph{if there exist an immersion} then the bundle of Smale-Hirsch is always rationally trivial (by Corollary~\ref{corodecoro} and Lemma~\ref{federico}).\\
When $k$ is even the result is slightly different. In this case we need some extra hypothesis  to  get an easy description of  the connected component of $Imm\left(M,\mathbb{R}^{m+k}\right)$. In order to state this precisely, recall that for a real vector bundle $\xi$ over a space $B$,  the dual Pontryagin  classes of $\xi$ are the unique elements $1, \tilde{p}_1\left(\xi\right) \cdots$,  with $\tilde{p}_i\left(\xi\right)\in H^{4i}\left(B,\mathbb{Q}\right) \forall i\geq 1$, satisfying:
 \[
 \left(1+p_1\left(\xi\right)+\cdots+p_{\lfloor\frac{m-1}{2}\rfloor}\left(\xi\right)\right)\left(1+\tilde{p}_1\left(\xi\right)+\cdots\right)=1
 \] 
 where $p_1\left(\xi\right),\cdots, p_{\lfloor\frac{m-1}{2}\rfloor}\left(\xi\right)$ are the ordinary Pontryagin classes of $\xi$. In other words $\tilde{p}_i\left(\xi\right)= {p}_i\left(-\xi\right)$ where $-\xi$ is the opposite virtual bundle of $\xi$. Also we denote by $e\left(\xi\right)$ the Euler class of $\xi$.\\
Here we use the notation $\lfloor\;\rfloor$ and $\lceil \;\rceil$ for the floor and ceil operations: 
\[ \lfloor{x}\rfloor=\max\{n\in \mathbb{Z}\mid n\leq x \} \text{ and }  \lceil{x}\rceil=\min\{n\in \mathbb{Z}\mid x\leq n \}
. \]We prove the following result
\begin{theorem}[Theorem~\ref{caseven}]
\label{theoremc}
Let $M$ be a manifold of dimension $m\geq0$, simply connected and of finite type, and let $k\geq 2$ be an even integer. Assume that the following two statement hold
\begin{itemize}
\item $e\left(\tau_M\right)=0,$ and
\item
$\tilde{p}_{\frac{k}{2}}\left(\tau_M\right)=0$. \end{itemize}
Then each connected component of $Imm\left(M,\mathbb{R}^{m+k}\right)$ has the rational homotopy type of a product of  some component of  the mapping space of $M$ into $S^k$,$Map\left(M,S^k\right)$, with a product of Eilenberg-Mac Lane spaces which only depends on $m$, $k$ and the rationals Betti numbers of $M$.\\
If moreover $H^k\left(M,\mathbb{Q}\right)=0$ then all the components of $Imm\left(M,\mathbb{R}^{m+k}\right)$ have the same rational homotopy type as well as all the component of $Map\left(M,S^k\right)$.
\end{theorem}

Theorem~\ref{theoremB}  and Theorem~\ref{theoremc} imply that the Betti numbers of each connected component  of the space $Imm\left(M, \mathbb{R}
^{m+k}\right)$ have polynomial growth. By combining these results and the main result of \cite{alp} we deduce that, when $k\geq m+1$ and if  $\chi\left(M\right)\leq -2$, then the rational Betti numbers of the space of embeddings  $Emb\left(M,\mathbb{R}^{m+k} \right)$ have exponential growth.\\ When $M$ is a compact manifold without boundary and $k\geq \frac{m}{2}+1$, we give an explicit formulas of the ranks of homotopy groups of the connected components of space of immersions of $M$ into $\mathbb{R}^{m+k}$(see Theorem~\ref{seriepoincare4} and Theorem~\ref{seriepoincare2}). 

We consider now the case of a manifold $N$ more general than $\mathbb{R}^{m+k}$. We will always assume that all the Pontryagin classes of $N$ vanish. We prove that the path connected of $Imm\left(M,N\right)$ has the rational homotopy type of some component of an explicit mapping space when the Pontryagin classes of $N$ vanishes .
Before starting our main  result in this case, we fix some notations. 
\begin{itemize} 
\item
If $X$ and $Y$ and if $\phi\colon X\rightarrow Y$ is a continuous map then we denote by  $Map\left(X,Y;\phi\right)$  the path-connected component containing $\phi$ of  the space of continuous map from $X$ to $Y$, $Map\left(X,Y\right)$ equipped with  the compact-open topology . 
\item
If $f\colon M\looparrowright N$ is an immersion then
  $Imm\left(M,N;f\right)$ is the connected component of  $Imm\left(M,N\right)$ containing $ f$. 
    \end{itemize} 
We prove the following result:
\begin{theorem}[Theorem~\ref{theorem1} ]
  \label{theoremA}
 Suppose that  $k\geq 2$. Let $f\colon M \looparrowright N$ be an immersion. Suppose that all the Pontryagin classes of $N$ vanish. If one of the following statement hold:
 \begin{itemize}
 \item
 $k$ is odd, or
 \item
  the Euler class of $M$ is zero and  $\tilde{p}_{\frac{k}{2}}\left(\tau_M\right)=0$
   \end{itemize}
   then we have the following rational homotopy equivalence:
  \begin{eqnarray*}
 \label{eq:img}
 Imm\left(M,N;f\right) \simeq_{\mathbb{Q}}Map\left(M,V_m\left(\tau_N\right);\phi_f\right),
 \end{eqnarray*}
 where  $\phi_f$ is some continuous  map from $M$ to $V_m\left(\tau_N\right)$ (see Section~\ref{niger} for the definition of the space $V_m\left(\tau_N\right)$).
 \end{theorem}

 \textbf{Plan of the paper}
\begin{itemize}
\item
In Section~\ref{basicnotions} we recall some definitions and results of rational homotopy theory which are important in this paper.
\item 
In Section~\ref{constructionofstiefel} we introduce the construction of Stiefel associated to a pair of vector bundles and we give some nice properties  of this construction. The importance of this construction is that the space $Imm\left(M,N\right)$ is homotopy equivalent to a space of sections associated to the construction of Stiefel of the tangent bundle $\tau_M$ and $\tau_N$
\item
In Section~\ref{ratmod} we construct a rational model of the construction of Stiefel.
\item 
In Section~\ref{modelimm} using our rational model of the Stiefel construction we prove that it is rationally trivial under good hypothesis, from which we derive easily its space of sections. The  main results  of this paper will follow directly.
\item 
In Section~\ref{computation} we   compute the ranks of the homotopy groups of the connected components of the space of immersions $Imm\left(M,\mathbb{R}^{m+k}\right)$.
\end{itemize}
\textbf{Acknowlement}
The present work is a part of my PhD thesis. I would like to thank my advisor, Pascal Lambrechts, for making it possible throughout his advices and encouragements. I thank Yves F\'elix for suggesting that Lemma~\ref{yvess} is true without a codimension condition. I also thank Thomas Goodwillie and Frederico Cantero Moran for many discussions. 
\section{Rational homotopy theory}
\label{basicnotions}
In this paper  we will use the standard tools and results of rational homotopy theory, following the notation and terminology of \cite{rht}. Thus, the rationalization of a simply connected space $X$ is a rational simply connected space $X_{\mathbb{Q}}$  together with a map $r\colon X\rightarrow X_{\mathbb{Q}}$ such that $\pi_\ast\left(r\right)\otimes\mathbb{Q}$ is an isomorphism.  If $V$ is a graded vector space over $\mathbb{Q}$, we denote by $\Lambda {V}$ the free commutative graded algebra generated by $V$. Recall that  $A_{pl}$  is the Sullivan-de Rham contravariant functor  which associates to each space $X$  a commutative  graded differential algebra (cgda) $A_{pl}\left(X\right)$. A cgda $\left(A,d_A\right)$ is a \emph{model} of $X$ if there exist a zig-zag of quasi-isomorphism between $\left(A,d_A\right)$ and $A_{pl}\left(X\right)$. Similarly a cgda morphism $\left(B,d_B\right)\rightarrow \left(A,d_A\right)$ is a \emph{model} of the map of spaces $f\colon X\rightarrow Y$ if there exist a zig-zag of quasi-isomorphisms between $\left(B,d_B\right)\rightarrow \left(A,d_A\right)$ and $A_{pl}\left(f\right)\colon A_{pl}\left(Y\right)\rightarrow A_{pl}\left(X\right)$.  For each fibre bundle $F\rightarrow E\rightarrow B$ with $E$ and $B$ simply connected and one of the graded spaces $H_\ast\left(B,\mathbb{Q}\right),H_\ast\left(F,\mathbb{Q}\right) $ has finite type, there exist a commutative diagram 
\begin{eqnarray*}
\xymatrix{A_{pl}\left(X\right)\ar[r] &A_{pl}\left(E\right)\ar[r]& A_{pl}\left(F\right)\\
\left(A,d_A\right)\ar@{^{(}->}[r]\ar[u]^{m}&\left(A\otimes \Lambda W,D\right)\ar[r]\ar[u]^{\phi}&\left(\Lambda W,\bar{D}\right)\ar[u]^{\bar{\phi}}
}
\end{eqnarray*}
where $m, \phi$ and $\bar{\phi}$ are quasi-isomorphisms (see \cite[Theorem~15.3]{rht}).
The inclusion $\left(A,d_A\right)\hookrightarrow \left(A\otimes \Lambda W,D\right)$ is called  a \emph{relative Sullivan model} of the fibre bundle $F\rightarrow E\rightarrow B$.
 \begin{defn}
\label{rt}
Let $F\rightarrow E\rightarrow B$ be a fibre bundle between simply connected spaces a relative Sullivan model $
\left(A,d_A\right)\hookrightarrow \left(A\otimes \Lambda W, D\right) 
$. Then $F\rightarrow E\rightarrow B$ is said to be \emph{rationally trivial } if there exist an isomorphism
\[
\psi\colon \left(A,d_A\right)\otimes \left(\Lambda W,\bar{D} \right)\overset{\cong}\longrightarrow  \left(A\otimes \Lambda W, D\right)
\]
such that the following diagram  commute 

\[
 \xymatrix{ \left(A,d_A \right) \ar@{^{(}->}[rr] \ar@{^{(}->}[rd] && \left(A\otimes \Lambda W,D\right)  \\ &\left(A,d_A\right)\otimes \left(\Lambda W,\bar{D}\right).\ar[ru]_{\psi}^{\cong}} 
 \]
\end{defn}
\begin{rmq}
If $ F\rightarrow E\rightarrow B$ is rationally trivial, then its rationalization $ F_{\mathbb{Q}}\rightarrow E_{\mathbb{Q}}\rightarrow B_{\mathbb{Q}}$ is  a trivial fibration.
\end{rmq}

\begin{defn}
Let $F\rightarrow E\overset{\xi}\rightarrow X$ be a fibre bundle.  A \emph{fibrewise rationalization} of $\xi$ is a fibre bundle $F'\rightarrow E'\overset{\xi'}\rightarrow X$ and a map $g\colon E\rightarrow E'$ over $X$ such that $g$ induces a rationalization $F\rightarrow F'$ on the fibres.
\end{defn}
\begin{thm}\cite[Corollary~6.2]{irene}
\label{irene}
Every fibre bundle $F\rightarrow E\overset{\xi}\rightarrow X$ with $F$ simply connected has a fiberwise rationalization which is unique up to a homotopy equivalence.
\end{thm}
\begin{nota} The fiberwise rationalization of a fibre bundle $F\rightarrow E\overset{\xi}\rightarrow X$ will be noted $F_{\mathbb{Q}}\rightarrow E_{\left(\mathbb{Q}\right)}\overset{\xi_{\left(\mathbb{Q}\right)}}\rightarrow X$.
\end{nota}
From Theorem~\ref{irene} we deduce  the following corollary.
\begin{coro}
\label{corodecoro1}
Let $ F\rightarrow E\overset{\xi}\rightarrow B$  be a fibre bundle.
If its rationalization  $ F_{\mathbb{Q}}\rightarrow E_{\mathbb{Q}}\overset{\xi_{\mathbb{Q}}}\rightarrow B_{\mathbb{Q}}$ is trivial, then its fiberwise rationalization $ F_{\mathbb{Q}}\rightarrow E_{\left(\mathbb{Q}\right)}\overset{\xi_{\left(\mathbb{Q}\right)}}\rightarrow B$ is trivial.
\end{coro}
The importance of this corollary will appear in the proof of main result of this paper (Theorem~\ref{theorem1}).

\section{Construction of Stiefel associated to a couple of vector bundles}
\label{constructionofstiefel}
In this section we introduce the construction of Stiefel associated to a couple of vector bundles. This construction comes with two fibre bundles. The importance of this construction in the study of the space of immersions, of a manifold $M$ in a manifold $N$, comes from the Smale-Hirsch theorem, which identifies the space of immersions with the space of sections of one of the fibre bundles of the construction of Stiefel associated to the tangent bundles of $M$ and $N$.
\subsection{Construction}
\label{niger}
In all this paper for an inner product space $W$ and an integer $m\geq  0$, we denoted by $V_m\left(W\right)$  the oriented Stiefel manifold of $m$ oriented frames in $W$ and by$G_m\left(W\right)$  the oriented grassmannian of $m$ planes in $W$. Thus we have a bundle
\begin{eqnarray}
\label{ak1}
GL^+\left(m\right)\longrightarrow V_m\left(W\right)\longrightarrow G_m\left(W\right)
\end{eqnarray}
 Note that $W$ can be an infinite dimensional  vector space, in which case (\ref{ak1}) is the universal principal bundle classifying  principal $GL^+\left(m\right)$-bundle..\\

Let $\eta$ be an oriented vector bundle of rank $m+k$  over $Y$. We have an associated fibre bundle 
\begin{eqnarray}
\label{framedbundle}
 V_m\left(\mathbb{R}^{m+k}\right)\longrightarrow V_m\left(\eta\right)\overset{\pi}\longrightarrow Y,
\end{eqnarray}
with total space
\begin{eqnarray*}
V_m\left(\eta\right)=\{\left(y,v\right)| y\in Y,\text{ and } v=\left(v_1,\cdots,v_m\right) \text{ is an oriented $m$-frame in } E_y\eta \},
\end{eqnarray*}
where $E_y\eta$ is the fibre of $\eta$ over $y\in Y$, and the projection $\pi$ is defined by 
\[
\pi\left(y,v\right)=y.
\]
When $k=0$, then $V_m\left(\mathbb{R}^m\right)=GL^+\left(m\right)$ and $V_m\left(\eta\right)$ is the $GL^+\left(m\right)$-principal bundle 
\[
 GL^+\left(m\right)\longrightarrow \prin\left(\eta\right)\longrightarrow Y
\]
associated to the vector bundle $\eta$.\\
Now we  introduce the \emph{construction of Stiefel} associated to a couple of vector bundle $\xi$ and $\eta$, where 
$\xi\colon \mathbb{R}^m\rightarrow E\xi\rightarrow X$ is an oriented vector bundle of rank $m$ and $\eta\colon \mathbb{R}^m\rightarrow E\eta\rightarrow Y$ is an oriented vector bundle of rank $m+k$.   We consider the $GL^+\left(m\right)$-principal bundle associated to $\xi$
\[
GL^+\left(m\right)\longrightarrow \prin\left(\xi\right)\longrightarrow X
\]
and the $m$-framed bundle (\ref{framedbundle}) associated to $\eta$
 \[
  V_m\left(\mathbb{R}^{m+k}\right)\longrightarrow V_m\left(\eta\right)\overset{\pi}\longrightarrow Y.
 \]
 The group $GL^+\left(m\right)$ acts on $\prin\left(\xi\right)$ on the right  and if $k\geq 1$ it acts  on $V_m\left(\eta\right)$ on the left (for more about this action see \cite[pp~263-264]{hirsh}). These actions induce a diagonal action on the product $ \prin\left(\xi\right)\times V_m\left(\eta\right)$.  Consider  $\prin\left(\xi\right)\times_{GL^+\left(m\right)} V\left(\eta\right)$ the space of orbits of $\prin\left(\xi\right)\times V_m\left(\eta\right)$  under the diagonal action. Set 
\[
\Stieffel\left(\xi,\eta\right):=\prin\left(\xi\right)\times_{GL
^+\left(m\right)} V_m\left(\eta\right).
\]
We have the following diagram of fibre bundles
\begin{eqnarray}
\label{constructiondestiefel}
\xymatrix{
\Stieffel\left(\xi,\eta\right)\ar[r]^-{\pi_2} \ar[d]_{\pi_1}&Y\\
X&
}
\end{eqnarray} 
where the maps $\pi_1$ and $\pi_2$ are defined for all  $\left[\left(x,u\right),\left(y,v\right)\right]$ in $\Stieffel\left(\xi,\eta\right)$ by
\begin{eqnarray*}
\pi_1\left[\left(x,u\right),\left(y,v\right)\right]=x \text{ and }\pi_2\left[\left(x,u\right),\left(y,v\right)\right]=y.
\end{eqnarray*}
Note that $\pi_1$ is a fibre bundle with fibre $V_m\left(\eta\right)$.
\begin{defn}
Let $\xi\colon  \mathbb{R}^m\rightarrow E\xi\rightarrow X$ and  $\eta\colon\mathbb{R}^{m+k}\rightarrow E\eta \rightarrow Y$ be two oriented vector bundles of  rank $m$ and $m+k$ respectively. Then:
\begin{itemize}
\item
the diagram (\ref{constructiondestiefel}) is called the \emph{construction of Stiefel} associated to $\xi$ and $\eta$,
\item
the fibre bundle
\[
V_m\left(\eta\right)\longrightarrow \Stieffel\left(\xi,\eta\right)\overset{\pi_1}\longrightarrow X
\]
is called the \emph{Stiefel bundle} associated to $\xi$ and $\eta$.
\end{itemize}
\end{defn}
In the following proposition, we  give a nice property of the construction of Stiefel which is the property of double functoriality. 
\begin{prop}
\label{prop25}
Let  $ \xi\colon \mathbb{R}^m\longrightarrow E\xi \longrightarrow X$ and $\;\; \eta\colon  \mathbb{R}^{m+k}\longrightarrow E\eta \longrightarrow Y$ be two oriented vector  bundles of rank $m$ and $m+k$ respectively. Let $f\colon X'\longrightarrow X$ and $g\colon Y'\longrightarrow Y$ be two continuous maps and let $f^\ast\xi$ and $g^\ast\eta$  be the pullback of $\xi$ and $\eta$ along $f$ and $g$ respectively. Then we have the following diagram in which (1), (2), (3) are pullbacks diagrams 
\begin{eqnarray}
\label{diagrammedp}
 \def\commutatif{\ar@{}[rd]|{\circlearrowleft}}
   \xymatrix{
 \Stieffel\left(f^\ast\xi,g^\ast\eta\right) \ar@/^2pc/[rrrr]^{\pi'_2}  \ar@/_4pc/[dd]_{\pi'_1} \ar[rr]\ar[d] \ar@{}[rd]^-{p.b.(1)}&& \Stieffel\left(\xi,g^\ast\eta\right) \ar[rr] \ar[d] \ar@{}[rd]^-{p.b.(2)}&&Y'\ar[d]^{g}\\
 \Stieffel\left(f^\ast\xi,\eta\right)  \ar[d]\ar[rr]\ar@{}[rd]^-{p.b.(3)}&& \Stieffel\left(\xi,\eta\right)  \ar[rr]_-{\pi_2}\ar[d]^{{\pi_1}}&& Y \\
    X'\ar[rr]_{f} && X& }
  \end{eqnarray}
where $\pi'_1$ and $\pi'_2 $ are  the  projections of construction of Stiefel associated to $f^\ast\xi$ and $g^\ast\eta$. 
\end{prop}
The rest of this section is devoted to the proof of this theorem. We start with Diagram~(3), the bottom square of (\ref{diagrammedp}).
   Let 
 \[
GL^+\left(m\right)\longrightarrow \prin\left(f^\ast\xi\right)\longrightarrow X'
\]
be the $GL^+\left(m\right)$-principal bundle associated to $f^\ast\xi$. A typical  element of $\prin\left(f^\ast\xi\right)$ is a couple $\left(x',u\right)$ with $x'\in X'$ and $u$ an $m$-frame of the fibre of $f^\ast\xi$ over $x'$. We identify this fibre with the fibre of $\xi$ over $f\left(x'\right)$ and clearly we have a map $\prin\left(f\right)$ from $\prin\left(f^\ast\xi\right)$ to $\prin\left(\xi\right)$ defined by 
\[
\prin\left(f\right)\left(x',u\right)=\left(f\left(x'\right),u\right).
\]
This map is $GL^+\left(m\right)$-equivariant. Consequently the map
\[
\prin\left(f\right)\times id \colon \prin\left(f^\ast\xi\right)\times V_m\left(\eta\right)\longrightarrow \prin\left(\xi\right) \times V_m\left(\eta\right)
\]
is  $GL^+\left(m\right)$-equivariant with diagonal action.  After pass on the orbits we obtain the map 
\[
\prin\left(f\right)\times_{GL^+\left(m\right)} id \colon \prin\left(f^\ast\xi\right)\times_{GL^+\left(m\right)} V_m\left(\eta\right)\longrightarrow \prin\left(\xi\right) \times_{GL^+\left(m\right)} V_m\left(\eta\right). 
\]
We have the following commutative diagram:
\begin{eqnarray}
\label{diagram1}
\xymatrix{
\prin\left(f^\ast\xi\right)\times_{GL^+\left(m\right)} V_m\left(\eta\right) \ar[rr]^{ \prin\left(f\right)\times_{GL^+\left(m\right)} id} \ar[d]&&\prin\left(\xi\right)\times_{GL^+\left(m\right)} V_m\left(\eta\right)\ar[d]\\
X'\ar[rr]^{f}&&X
}
\end{eqnarray}
that identifies to the diagram (3) of diagram (\ref{diagrammedp}).
\begin{lem} 
\label{prop2}
The diagram~(\ref{diagram1}) is a pullback.
\end{lem}
\begin{proof} Since
 \[
 \prin\left(f^\ast\xi\right)\cong f^\ast \prin\left(\xi\right),
 \]
 then
  \begin{eqnarray}
  \label{eqq}
 \prin\left(f^\ast\xi\right)\times_{GL^+\left(m\right)}V_m\left(\eta\right)\cong f^\ast \prin\left(\xi\right)\times_{GL^+\left(m\right)}V_m\left(\eta\right)
 \end{eqnarray}
 by \cite[Proposition~6.3, p 47]{huss} we have
 \begin{eqnarray}
 \label{eqqq}
 f^\ast \prin\left(\xi\right)\times_{GL^+\left(m\right)}V_m\left(\eta\right)\cong f^\ast\left(\prin\left(\xi\right)\times_{GL^+\left(m\right)}V_m\left(\eta\right)\right)
   \end{eqnarray}
   by (\ref{eqq} ) and (\ref{eqqq} ) we deduce
   \[
   \prin\left(f^\ast\xi\right)\times_{GL^+\left(m\right)}V_m\left(\eta\right)\cong f^\ast\left( \prin\left(\xi\right)\times_{GL^+\left(m\right)}V_m\left(\eta\right)\right).
   \]
\end{proof}
In the following lemma we prove that the diagram (2) of (\ref{diagrammedp}) is a pullback diagram.
 \begin{lem}
 \label{prop327}
 The following diagram is a pullback 
 \begin{eqnarray}
\label{diagramme325}
\xymatrix{
\prin\left(\xi\right)\times_{GL^+\left(m\right)}V_m\left(g^\ast\eta\right)\ar[rr]^{{id}\times_{GL^+\left(m\right)}V_m\left({g}\right)}\ar[d]_{\pi_2'}&&\prin\left(\xi\right)\times_{GL^+\left(m\right)}V_m\left(\eta\right)\ar[d]^{\pi_2}\\
Y'\ar[rr]_{g}&&Y,
}
\end{eqnarray}
where
\begin{itemize}
\item
for all  $\left[\left(x,u\right),\left(y,v\right)\right]\in \prin\left(\xi\right)\times_{GL^+\left(m\right)} V_m\left(g^\ast\eta\right)$  we have
\[
\pi_2\left[\left(x,u\right),\left(y,v\right)\right]=y
\]
\item
for all  $\left[\left(x,u\right),\left(y,v\right)\right]\in \prin\left(\xi\right)\times_{GL^+\left(m\right)} V_m\left(g^\ast\eta\right)$  we have
\begin{eqnarray*}
\pi'_2\left[\left(x,u\right),\left(y,v\right)\right]&=&y \text{ and }\\
id\times_{GL^+\left(m\right)}\V_m\left({g}\right)\left[\left(x,u\right),\left(y,v\right)\right]&=&\left[\left(x,u\right),\left({g}\left(y\right),v\right)\right].
\end{eqnarray*}
\end{itemize}
  \end{lem}

    \begin{proof}
  Consider  the following diagram
  \[
  \xymatrix{
  \prin\left(\xi\right)\times_{GL^+\left(m\right)}V_m\left(g^\ast\eta\right)\ar[rr]^{{id}\times_{GL^+\left(m\right)}\V_m\left({g}\right)}\ar[d]_{p_1'}\ar@/_3pc/[dd]_{\pi'_2}\ar@{}[rrd]^{(\MakeUppercase{\romannumeral 1})}&&\prin\left(\xi\right)\times_{GL^+\left(m\right)}V_m\left(\eta\right)\ar[d]^{p_1}\ar@/^3pc/[dd]^{\pi_2}\\
\V_m\left(g^\ast\eta\right)\ar[rr]^{V_m\left(g\right)}\ar[d]_{p_2'}\ar@{}[rrd]^{(\MakeUppercase{\romannumeral 2})}&&V_m\left(\eta\right)\ar[d]^{p_2}\\  
Y'\ar[rr]_{g}&&Y
}
  \]
   The diagram $(\MakeUppercase{\romannumeral 2})$ is a pullback (see \cite[Section~1.5.3]{piccione} for the proof). To prove that the big diagram is a pullback, consider  the space $Z$ and two continuous maps
\[
 \alpha\colon Z\longrightarrow Y', \text{ and  }
 \beta=\left(\beta_1,\beta_2\right)\colon Z\longrightarrow \prin\left(\xi\right)\times_{GL^+\left(m\right)}V_m\left(\eta\right)
 \] such that 
 \[
 \pi_2\circ\beta=g\circ\alpha.
 \]
 We will prove that there exist an unique map 
 \[
\Phi\colon Z\longrightarrow \prin\left(\xi\right)\times_{GL^+\left(m\right)}V_m\left(g^\ast\eta\right) 
\]
such that
\[
\pi'_2\circ\Phi=\alpha \text{ and }{id}\times_{GL^+\left(m\right)}V_m\left({g}\right)\circ \Phi=\beta.
\]
\textbf{Existence}: Since $\pi_2=p_2\circ p_1$, then from 
 \[
 \pi_2\circ\beta=g\circ\alpha,
 \]
 we have
 \begin{eqnarray}
 \label{22}
 p_2\circ \left(p_1\circ\beta\right)=g\circ\alpha.
 \end{eqnarray}
 Because $(\MakeUppercase{\romannumeral 2})$ is a pullback, the equation (\ref{22}) involves the existence of an unique map
\[
\Phi_2\colon Z\longrightarrow V_m\left(g^\ast\eta\right)
\]
such that
\[
p'_2\circ \Phi_2=\alpha \text{ et } V_m\left(g\right)\circ \Phi_2=p_1\circ \beta.
\]
Let $\Phi=\left(\beta_1,\Phi_2\right)$. The map $\Phi$ is well defined and for all  $z\in Z$ we have on the one hand:
\begin{eqnarray*}
\pi'_2\circ\Phi\left(z\right)=p_2'\circ p_1'\circ\Phi\left(z\right)&=&p'_2\circ p_1'\left[\beta_1\left(z\right),\Phi_2\left(z\right)\right]\\
&=&p_2'\circ \Phi_2\left(z\right)\\
&=&\alpha\left(z\right),
\end{eqnarray*}
and on the other hand
 \begin{eqnarray*}
{id}\times_{GL^+\left(m\right)}V_m\left({g}\right)\circ \Phi\left(z\right)&=&{id}\times_{GL^+\left(m\right)}V_m\left({g}\right)\left[\beta_1\left(z\right),\Phi_2\left(z\right)\right]\\
&=&\left[\beta_1\left(z\right),V_m\left(g\right)\circ \Phi_2\left(z\right)\right].
\end{eqnarray*}
 Since $V_m\left(g\right)\circ \Phi_2=p_1\circ \beta$,  we have
\[V_m\left(g\right)\circ\Phi_2\left(z\right)= \beta_2\left(z\right).\]
As conclusion we have:
\[
\pi'_2\circ\Phi=\alpha \text{ and }{id}\times_{GL^+\left(m\right)}\V_m\left({g}\right)\circ \Phi=\beta.
\]
\textbf{Unicity}: Suppose that there exist another map 
\[
\Phi'=\left(\Phi_1',\Phi_2'\right)\colon  Z\longrightarrow \prin\left(\xi\right)\times_{GL^+\left(m\right)}V_m\left(g^\ast\eta\right) 
\]
such that
\[
\pi'_2\circ\Phi'=\alpha \text{ and }{id}\times_{GL^+\left(m\right)}V_m\left({g}\right)\circ \Phi'=\beta,
\]
then for all $z\in Z$ we have:
\begin{eqnarray}
\label{00}
\pi'_2\circ\Phi'\left(z\right)=p_2'\circ p_1'\circ\Phi'\left(z\right)&=&p'_2\circ p_1'\left[\Phi_1'\left(z\right),\Phi'_2\left(z\right)\right]\nonumber\\
&=&p_2'\circ\Phi'_2\left(z\right)=\alpha\left(z\right),
\end{eqnarray}
and
\begin{eqnarray}
\label{99}
&&V_m\left(g\right)\circ p'_1\circ \Phi'\left(z\right)\nonumber\\
&=&V_m\left(g\right)\circ\Phi'_2\left(z\right)=\beta_2\left(z\right) .
\end{eqnarray}
Since $(\MakeUppercase{\romannumeral 2})$ is a pullback, from (\ref{00}) and (\ref{99}) we have:
\[
\Phi_2=\Phi_2'. 
\]
It remains to prove that 
\[
\Phi'_1=\beta_1. 
\]
Because
\begin{eqnarray*}
{id}\times_{GL^+\left(m\right)}V_m\left({g}\right)\circ \Phi'\left(z\right)&=&{id}\times_{GL^+\left(m\right)}V_m\left({g}\right)\left[\Phi'_1\left(z\right),\Phi'_2\left(z\right)\right]\\
&=&\left[\Phi'_1\left(z\right),V_m\left(g\right)\circ \Phi'_2\left(z\right)\right]\\
&=&\left[\Phi'_1\left(z\right),V_m\left(g\right)\circ \Phi_2\left(z\right)\right]\\
&=&\left[\Phi'_1\left(z\right), \beta_2\left(z\right)\right]
\end{eqnarray*}
from
${id}\times_{GL^+\left(m\right)}V_m\left({g}\right)\circ \Phi=\beta$
we have that 
\[
\left[\Phi'_1\left(z\right), \beta_2\left(z\right)\right]=\left[ \beta_1\left(z\right),\beta_2\left(z\right)\right],
\]
therefore there exist  $g\in GL^+\left(m\right)$ such that
\[
\left(\Phi'_1\left(z\right), \beta_2\left(z\right)\right)=\left( \beta_1\left(z\right)g,g^{-1}\beta_2\left(z\right)\right).
\]
Since the action of $GL^+\left(m\right)$ sur $V_m\left(\eta\right)$ is free, we have:
\[
\beta_2\left(z\right)=g^{-1}\beta_2\left(z\right)\Rightarrow g=1
\] 
thus,
\[
\left(\Phi'_1\left(z\right), \beta_2\left(z\right)\right)=\left( \beta_1\left(z\right),\beta_2\left(z\right)\right).
\]
\end{proof}
\begin{lem}
\label{lem7}
The diagram
 \begin{eqnarray}
 \xymatrix{
 \Stieffel\left(f^\ast\xi,g^\ast\eta\right)\ar@{}[rd]^-{pb}\ar[r]\ar[d]&\Stieffel\left(\xi,g^\ast\eta\right)\ar[d]\\
 \Stieffel\left(f^\ast\xi,\eta\right)\ar[r]&\Stieffel\left(\xi,\eta\right)
 }
 \end{eqnarray} 
 is a pullback.
 \end{lem}
 \begin{proof}
 Consider the diagram 
 \begin{eqnarray*}
 \xymatrix{
 \Stieffel\left(f^\ast\xi,g^\ast\eta\right)\ar[r]\ar[d]\ar@{}[rd]^{(1)}&\Stieffel\left(\xi,g^\ast\eta\right)\ar[d]\\
  \Stieffel\left(f^\ast\xi,\eta\right)\ar[r]\ar[d]  \ar@{}[rd]^{(2)}&\Stieffel\left(\xi,\eta\right)\ar[d]\\
X' \ar[r]& X}
 \end{eqnarray*}
By Lemma~\ref{prop2} the part (2) and the big diagram are  pullbacks . The conclusion comes from \cite{maclane} (see also\cite[Lemma~1.1]{Lacks}). 
  \end{proof}

\begin{proof}[Proof of Proposition~\ref{prop25}]
It follows from Lemma~\ref{prop2},~\ref{prop327} and \ref{lem7}.
\end{proof}

\subsection{Universal construction of Stiefel}

In this section we study the construction of Stiefel  associated to the universal vector bundles. \\
 Recall that for $m\geq 0$, the universal vector bundle 
\[
 \mathbb{R}^m\longrightarrow \gamma^m \longrightarrow G_m\left(\mathbb{R}^\infty\right),
\]
is a vector bundle of rank $m$ where
\[
\gamma^m=\{\left(W,x\right) \mid\; W \text{ is an $m$-vector subspace of } \mathbb{R}^\infty,\;\dim\; W=m,\text{ and } x\in W\},
\]
and the projection is the map which send $\left(W,x\right)$ to $W$. We will prove the following 
\begin{prop}
   \label{equivalencebimk} 
    For $m,k\geq 0$ there exists a commutative diagram
       \begin{eqnarray}
\label{diagramme336A}
\xymatrix{
G_m\left(\mathbb{R}^\infty\right)&G_m\left(\mathbb{R}^\infty\right)\times G_{k}\left(\mathbb{R}^{\infty}\right)\ar[l]_-{{pr}_1}\ar[r] ^-{\rho_{m,k}}\ar[d]^{\beta}_{\simeq}&G_{m+k}\left(\mathbb{R}^\infty\right) \\
G_m\left(\mathbb{R}^\infty\right)\ar@{=}[u]&\Stieffel\left(\gamma^m,\gamma^{m+k}\right)\ar[r]_-{\pi_2}\ar[l]^-{\pi_1}&G_{m+k}\left(\mathbb{R}^\infty\right),\ar@{=}[u]}
\end{eqnarray} 
 in which   
 \begin{enumerate}
 \item
the map $\beta$ is a homotopy equivalence,
\item
the bottom line is the construction of Stiefel associated to the universal vector bundles $\gamma^m$ and $\gamma^{m+k}$.
\item
$pr_1$ is the projection on the first factor and $\rho_{m,k}$ sends $\left(W,U\right)$ to $W\oplus U\subset \mathbb{R}^\infty\oplus \mathbb{R}^\infty\cong \mathbb{R}^\infty$
\end{enumerate}
\end{prop}
\begin{rmq} The top line of Diagram~\ref{diagramme336A} is equivalent to the zigzag
\[
BSO\left(m\right)\overset{pr_1}\leftarrow BSO\left(m\right)\times BSO\left(k\right)\overset{\tilde{\rho}}\rightarrow BSO\left(m+k\right)
\]
where $\tilde{\rho}$ is the universal map classifying the Whithney sum.
\end{rmq}
From  Proposition~\label{equivalencebimk}  and Proposition~\ref{prop25} we deduce the following theorem which is crucial for the construction of rational model of the construction of Stiefel.

 \begin{thm}
  \label{corollaire333}
Let  $m$ and $k$ two  positive integers. Let $\xi$ and $ \eta$ be two oriented vector bundles classified by $\; f\colon X\rightarrow G_m\left(\mathbb{R}^\infty\right)$ and $g\colon Y\rightarrow G_{m+k}\left(\mathbb{R}^\infty\right)$ respectively. Then the construction of Stiefel associated  to  $\xi$ and $\eta$ is given by the double pullback:
\begin{eqnarray*}
  \xymatrix{
\Stieffel\left(\xi,\eta\right)\ar[dd]_{\pi_1}\ar@/^2pc/[rr] ^{\pi_2} \ar[r]& \Stieffel\left(\gamma^m,\eta\right)\ar[d]\ar@{}[rd]^-{p.b.} \ar[r]\ar[d]&Y\ar[d]^{g}\\
\ar[d]\ar@{}[rd]^{p.b.}&G_m\left(\mathbb{R}^\infty\right)\times G_k\left(\mathbb{R}^\infty\right)\ar[r]_-{\rho_{m,k}} \ar[d]^{pr_1}& G_{m+k}\left(\mathbb{R}^\infty\right)\\
X \ar[r]_{f}&G_m\left(\mathbb{R}^\infty\right).
}
 \end{eqnarray*}
 \end{thm}
 The  end of this section is devoted for the proof of Proposition~\ref{equivalencebimk}. We start with the construction of the map $\beta$.\\
 Let $l\colon \mathbb{R}^\infty\oplus \mathbb{R}^\infty\rightarrow \mathbb{R}^\infty$ be the linear isomorphism defined by 
 \[
 l\left(\left(x_0,x_1,\cdots\right),\left(y_0,y_1,\cdots\right)\right)=\left(x_0,y_0,x_1,y_1,\cdots\right).
 \]
The isomorphism $l$ induces a homeomorphism between  $G_m\left( \mathbb{R}^\infty\oplus \mathbb{R}^\infty\right)$ and  $G_m\left(\mathbb{R}^\infty\right)$. Let two morphisms 
 $j_{odd}\colon \mathbb{R}^\infty\rightarrow  \mathbb{R}^\infty$
 and  $j_{even}\colon \mathbb{R}^\infty\rightarrow  \mathbb{R}^\infty$ defined by 
 \begin{eqnarray*}
 j_{odd}\left(x_0,x_1,\cdots\right)=\left(0,x_0,0,x_1,\cdots\right), \text{ and }j_{even}\left(x_0,x_1,\cdots\right)=\left(x_0,0,x_1,0,\cdots\right).
 \end{eqnarray*}
The morphisms $l, j_{odd}$ and $j_{even} $ induce the maps
\begin{eqnarray*}
Ej\colon V_m\left(\mathbb{R}^\infty\right)\times V_k\left(\mathbb{R}^\infty\right)&\longrightarrow &V_{m+k}\left(\mathbb{R}^\infty\right)\\
\left(u,v\right)&\mapsto&l\left(j_{odd}\left(u\right),j_{even}\left(v\right)\right),
\end{eqnarray*}
and
\begin{eqnarray*}
\rho_{m,k}\colon G_m\left(R^\infty\right)\times G_k\left(\mathbb{R}^\infty\right)&\longrightarrow &G_{m+k}\left(R^\infty\right)\\
\left(A,B\right)&\mapsto& l\left(j_{odd}\left(A\right)\oplus {j_{even}\left(B\right)}\right).
\end{eqnarray*}
Let $GL^+\left(m\right)\rightarrow V_m\left(\mathbb{R}^\infty\right)\overset{\pi_0}\rightarrow G_m\left(\mathbb{R}^\infty\right)$ be the universal $GL^+\left(m\right)$-principal, where the projection is defined by $\pi_0\left(e\right)=\langle{e}\rangle$ (where $\langle{e}\rangle$ is the oriented $m$-plane spanned by the oriented $m$ frame $e$).  By the definition of the principal bundle there is an open cover $\{O_i\}$ of $G_m\left(R^\infty\right)$ and continuous maps $\{\sigma_i\colon O_i\rightarrow V_m\left(\mathbb{R}^\infty\right)\}$, called locals sections over $O_i$, such that for all $i$ we have
\[
\pi_0\circ \sigma_i=id.
\]
Therefore, if $A$ is an element of $ G_m\left(R^\infty\right)$ there exists an open set  $O_i$ such that $A\in O_i$.Thus $\sigma_i\left(A\right)\in V_m\left(\mathbb{R}^\infty\right)$. For $\left(A,B\right)$ in $G_m\left(\mathbb{R}^\infty\right)\times G_k\left(\mathbb{R}^\infty\right)$, we define the map $\beta$ by:
\begin{eqnarray}
\label{equationbeta}
\beta\left(A,B\right)=\left[\sigma_i\left(A\right),\left(\rho_{m,k}\left(A,B\right),j_{odd}\left(\sigma_i\left(A\right)\right)\right)\right].
\end{eqnarray} 
The map $\beta$ is well defined. In fact suppose that there exist another open set $O_r$ containing $A$, and le $\sigma_r$ be the local section over $O_r$. Since 
$GL^+\left(m\right)\rightarrow V_m\left(\mathbb{R}^\infty\right)\rightarrow G_m\left(\mathbb{R}^\infty\right)$ is a principal bundle there exist $g$ in $GL^+\left(m\right)$ such that 
\[
\sigma_r\left(A\right)=\sigma_i\left(A\right).g.
\]
 Consequently we have
\begin{eqnarray*}
&&\left[\sigma_r\left(A\right),\left(\rho_{m,k}\left(A, B)\right),j_{odd}\left(\sigma_r\left(A\right)\right)\right)\right]=\left[\sigma_i\left(A\right).g,\left(\rho_{m,k}\left(A,B\right),j_{odd}\left(\sigma_i\left(A\right)g\right)\right)\right]\\
&=&\left[\sigma_i\left(A\right).g,g^{-1}.\left(\rho_{m,k}\left(A,B\right),j_{odd}\left(\sigma_i\left(A\right)\right)\right)\right]
=\left[\sigma_i\left(A\right),\left(\rho_{m,k}\left(A,B\right),j_{odd}\left(\sigma_i\left(A\right)\right)\right)\right].
\end{eqnarray*}

  \begin{lem} 
 \label{restriction}
For $A\in G_m\left(\mathbb{R}^\infty\right)$, the map 
\begin{eqnarray*}
\beta_A\colon G_k\left(\mathbb{R}^\infty\right)&\longrightarrow& V_m\left(\gamma^{m+k}\right)\\
B&\mapsto& \left(\rho_{m,k}\left(A, B\right),j_{odd}\left(\sigma_i\left(A\right)\right)\right)
\end{eqnarray*}
is a homotopy equivalence.
 \end{lem}
 \begin{proof}
 Consider the universal $SO\left(k\right)$-principal bundle 
 \[
 SO\left(k\right)\longrightarrow V_k\left(\mathbb{R}^\infty\right)\overset{\pi_0}\longrightarrow G_k\left(\mathbb{R}^\infty\right),
\]
where $\pi_0$ sends each $k$-frame of $\mathbb{R}^\infty$  on the $k$-plan that it spans
\[
\pi_0\left(v\right)=\langle{v}\rangle. 
\]
Consider the $SO\left(k\right)$-principal bundle
\[
SO\left(k\right)\longrightarrow V_{m+k}\left(\gamma^{m+k}\right)\overset{p_0}\longrightarrow V_{m}\left(\gamma^{m+k}\right)
\]
where

\[
V_{m+k}\left(\gamma^{m+k}\right)=\{\left(W,w\right) |\; W\subset \mathbb{R}^{\infty},\; \dim W= m+k, w \text{ is a } m+k- \text{frame of } W\}
\] 
and the projection $p_0$ is defined for all $\left(W,w\right)$ in $V_{m+k}\left(\gamma^{m+k}\right)$ by 
\[
p_0 \left(W,w\right)=\left(W, w_1\right),
\]
 where $w_1$ denotes the $m$ first vectors of the $m+k$-frame $w$. Let $A\subset \mathbb{R}^\infty$ be a subspace of dimension $m$. Let
\begin{eqnarray*}
\bar{\beta}_A \colon V_k\left(\mathbb{R}^\infty\right)&\longrightarrow& V_{m+k}\left(\gamma^{m+k}\right)\\
v&\mapsto&\left(\rho_{m,k}\left(A,\langle{ v\rangle}\right),Ej\left(\sigma_i\left(A\right), v\right)\right).
\end{eqnarray*}
The map $\bar{\beta}_A$ is a homotopy equivalence because $V_k\left(\mathbb{R}^\infty\right)$ and $ V_{m+k}\left(\gamma^{m+k}\right)$ are contractibles.  On the other hand for all $v$ in $V_k\left(\mathbb{R}^\infty\right)$, we have:
\begin{eqnarray*}
p_0\circ \bar{\beta}_A\left(v\right)&=&p_0\left(\rho_{m,k}\left(A, \langle{v\rangle}\right),Ej\left(\sigma_i\left(A\right),v\right)\right)\\
&=&\left(\rho_{m,k}\left(A,\langle{ v}\rangle\right),j_{odd}\left(\sigma_i\left(A\right)\right)\right),
\end{eqnarray*}
and
\begin{eqnarray*}
\beta_A\circ\pi_0\left(v\right)=\beta_A\left(\langle{v}\rangle\right)
= \left(\rho_{m,k}\left(A,{\langle{ v}\rangle}\right),j_{odd}\left(\sigma_i\left(A\right)\right)\right). 
\end{eqnarray*}
Therefore the diagram of $SO\left(k\right)$-principal bundles following is commutative.
\begin{eqnarray*}
\xymatrix{
V_k\left(\mathbb{R}^\infty\right)\ar[r]^-{\bar{\beta}_A}\ar[d]_{\pi_0}&V_{m+k}\left(\gamma^{m+k}\right) \ar[d]^{p_0}\\
G_k\left(\mathbb{R}^\infty\right) \ar[r]_-{\beta_A}&V_{m}\left(\gamma^{m+k}\right)}
\end{eqnarray*}
in other words $\left(\bar{\beta_A},\beta_A\right)$ is a  morphism of principal bundles. Since $\bar{\beta}_A$ is a homotopy equivalence, from the property of long suite exact associated to a fibre bundles we deduce that $\beta_A$ is a homotopy equivalence.
\end{proof}
  \begin{proof}[Proof  of Proposition~\ref{equivalencebimk}]
 The first part of the proof is easy. In fact, for $\left(A,B\right)$ in $G_m\left(\mathbb{R}^\infty\right)\times G_k\left(\mathbb{R}^\infty\right)$,
  we have on the one hand:
  \begin{eqnarray*}
\pi_1\circ\beta\left(A,B\right)&=&\pi_1\left(\left[\sigma_i\left( A\right),\left(\rho_{m,k}\left(A, B\right),j_{odd}\left(\sigma_i\left(A\right)\right)\right)\right]\right)\\
&=&\pi_0\left(\sigma_i\left(A\right)\right)=A
=pr_1\left(A,B\right),
\end{eqnarray*}
and on the other hand:
\begin{eqnarray*}
{\pi}_{2}\circ \beta\left(A,B\right)&=&{\pi}_{2}\left(\left[\sigma_i\left(A\right),\left(\rho_{m,k}\left(A,B\right),j_{odd}\left(\sigma_i\left(A\right)\right)\right)\right]\right)\\
&=&\rho_{m,k}\left(A,B\right).
\end{eqnarray*} 
this end the first part of the proof. \\
Now, we prove that  $\beta$ is a homotopy  equivalence. From the first part of of the proof we have 
\[
\pi_1\circ \beta =pr_1.
\]
Since 
\[
 pr_1\colon G_m\left(\mathbb{R}^\infty\right) \times G_k\left(\mathbb{R}^\infty\right) \longrightarrow G_m\left(\mathbb{R}^\infty\right),
 \]
 and
 \[
\pi_1\colon V_m\left(\mathbb{R}^\infty\right)\times_{GL^+\left(m\right)} V_m\left(\gamma^{m+k}\right)\longrightarrow G_m\left(\mathbb{R}^\infty\right)
\]
are fibre bundles, then $\left(\beta,id\right)$ is a morphism of fibre bundles. 
  Let  $A$ in $G_m\left(\mathbb{R}^\infty\right)$, the  restriction of the map $\beta$ on the fibres over $A$  is the map
  \[
  \beta_{\mid A} \colon G_k\left(\mathbb{R}^\infty\right)\longrightarrow \pi_1^{-1}\left(A\right)=\{\left[e,\left(W,w\right)\right] \mid\, \pi_1\left(e\right)=A\},
  \]
 defined  for all $B$ in $G_k\left(\mathbb{R}^\infty\right)$ by $ \beta_{\mid A} \left(B\right)=\beta\left(A,B\right)$. This map is the composition of 
 \[
 \beta_A\colon G_k\left(\mathbb{R}^\infty\right)\longrightarrow V_m\left(\gamma^{m+k}\right)
 \]
 which is a homotopy equivalence by Lemma~\ref{restriction},  and the map
 \[
 \alpha\colon V_m\left(\gamma^{m+k}\right)\longrightarrow \pi_1^{-1}\left(A\right),
 \] 
defined for all $\left(W,w\right)$ in $V_m\left(\gamma^{m+k}\right)$ by
\[
\alpha\left(W,w\right)= \left[\sigma_i\left(A\right),\left(W,w\right)\right]
\]
 which is a homeomorphism (see \cite[Proposition 1, p 198]{halperin2} or \cite[Proposition 3.7]{koke} for the proof). Consequently, from the property of long suite associated to a fibre bundles we deduce that $ \beta_{\mid A}$ is a homotopy equivalence.
 \end{proof}
 
\section{ Rational model of Stiefel bundle}
\label{ratmod}
In all this section we fix two oriented vector bundles $\xi \colon \mathbb{R}^m\rightarrow E\xi\rightarrow X$ and $\eta \colon\mathbb{R}^{m+k}\rightarrow E\eta\rightarrow Y$ 
 of rank $m$ and $m+k$ respectively. We  suppose that $X$ and $Y$ are simply connected. Our main goal is to construct a  cgda model of the Stiefel bundle (Theorem~\ref{modelfibredestiefel})
\begin{eqnarray}
\label{eqmds}
V_m\left(\eta\right)\longrightarrow \Stieffel\left(\xi,\eta\right)\overset{\pi_1}\longrightarrow X
\end{eqnarray} 
and to deduce that this bundle is rationally trivial under some mild hypothesis on the characteristics classes of $\xi$ (Corollary~\ref{corodecoro}). The data needed to build a cgda model of the Stiefel bundle are:
\begin{itemize}
\item
a cgda model $\left(A,d_A\right)$ of $X$,
\item
a  Sullivan model $\left(\Lambda V,d\right)$ of $Y$,
\item
 representatives of the dual  Pontryagin classes of $\xi$, $\tilde{p}_{\ell}\left(\xi\right)\in A^{4\ell}\cap\ker d_A,$  for each  integer $\ell \geq \frac{k}{2}$,
 \item
 representatives of the Pontryagin classes of $\eta$,  $p_{\ell}\left(\eta\right)\in \left(\Lambda{V}\right)^{4\ell}\cap\ker d$, for each  integer $\ell \geq \frac{k}{2}$,
 \item
representative of the Euler class of $\xi$, $e\left(\xi\right)\in A^{m}\cap\ker d_A$
\item
 representative of the Euler class  of $\eta$,   $e\left(\eta\right)\in \left(\Lambda V\right)^{m+k}\cap\ker d$ (if $m+k$ is even).
\end{itemize}
\begin{thm}
 \label{modelfibredestiefel} A relative Sullivan model of the  Stiefel bundle (\ref{eqmds}) is given by
 \begin{eqnarray*}
\label{hhhh}
\left(A,d_A\right)\hookrightarrow \left(A\otimes \Lambda V\otimes \Lambda\left(\{\bar{a}_{\ell}\}_{\frac{k}{2}\leq \ell\leq \frac{m+k-1}{2}},\bar{e}_{m+k},e_k\right),D\right)
\end{eqnarray*}
where the new generators are
\begin{itemize}
\item  
$e_k$ is a generator of degree $k$ that appears only when $k$ is even,
\item
$\bar{e}_{m+k}$ is a generator of degree $m+k-1$ that appears only when $m+k$ is even,
\item
$\bar{a}_{\ell}$ are generators of  degree $4\ell-1$ that appear for all integers $\ell$ between $\frac{k}{2}$ and $\frac{m+k-1}{2}$,

\end{itemize}
and the differential $D$ is given by
\begin{itemize}
\item
$D_{|A}=d_A,$ and 
$D_{\mid V}=d,$
\item
$D\bar{a}_{\frac{k}{2}}=1\otimes p_{\frac{k}{2}}\left(\eta\right)\otimes 1-1\otimes 1\otimes e_k^2- \tilde{p}_{\frac{k}{2}}\left(\xi\right)\otimes 1\otimes 1,$ if $k$ is even
\item
 $D\bar{a}_{\ell}=1\otimes p_{\ell}\left(\eta\right)\otimes 1- \tilde{p}_{\ell}\left(\xi\right)\otimes 1\otimes 1 \text{ if }$ for each integer $\ell\in \left[\frac{k}{2},\frac{m+k-1}{2}\right]$,
\item
$ D\bar{e}_{m+k}=1\otimes e\left(\eta\right)\otimes 1-e\left(\xi\right)\otimes{1}\otimes e_k$  if $m+k$ is even (with $e_k=0$ when $k$ is odd)
\item
$De_k=0$ if $k$ is even.
\end{itemize}
 \end{thm}
 \begin{rmq}
 Notice that $\deg\left(e_k\right)=k$ but $\deg\left(\bar{e}_{m+k}\right)=m+k-1$. Also note that the only Pontryagin classes $p_\ell\left(\eta\right)$ and dual Pontryagin classes $\tilde{p}_\ell\left(\xi\right)$ used in  this model are those corresponding to indices $\ell\in \left[\frac{k}{2},\frac{m+k-1}{2}\right]$.
 \end{rmq}
 From Theorem~\ref{modelfibredestiefel}  we deduce the following corollary which is very useful  in the construction of the rational model of the space of immersions.
 \begin{coro}
 \label{corodecoro}
 Let $k=rank \;\eta-rank\; \xi$. Assume that $\forall \ell\geq \frac{k}{2}, \tilde{p}_{\ell}\left(\xi\right)=0$.
  If  $k$ is odd or $e\left(\xi\right)=0$,  then the Stiefel bundle $\Stieffel\left(\xi,\eta\right)\overset{\pi_1}\rightarrow X$ is rationally trivial. 
 \end{coro}
 \begin{proof}
Under these hypotheses the differential $D$ of the model in Theorem~\ref{modelfibredestiefel} becomes
 \begin{itemize}
 \item
$D_{|A}=d_A,$
$D_{\mid V}=d,$
\item
$D\bar{a}_{\frac{k}{2}}=1\otimes p_{\frac{k}{2}}\left(\eta\right)\otimes 1-e_k^2$ and $D\bar{a}_{\ell}=1\otimes p_{\ell}\left(\eta\right)\otimes 1 \text{ if } \ell >\frac{k}{2}$,
\item
$ D\bar{e}_{m+k}=1\otimes e\left(\eta\right)\otimes 1.$
\end{itemize} 
Thus we have an isomorphism
\[
 \left(A\otimes \Lambda V\otimes \Lambda\left(\{\bar{a}_{\ell}\}_{\frac{k}{2}\leq \ell\leq \frac{m+k-1}{2}},\bar{e}_{m+k},e_k\right),D\right)\cong\left(A,d_A\right)\otimes\left(  \Lambda V\otimes \Lambda\left(\{\bar{a}_{\ell}\}_{\frac{k}{2}\leq \ell\leq \frac{m+k-1}{2}},\bar{e}_{m+k},e_k\right),\bar{D} \right).
 \]
 Therefore $\left(  \Lambda V\otimes \Lambda\left(\{\bar{a}_{\ell}\}_{\frac{k}{2}\leq \ell\leq \frac{m+k-1}{2}},\bar{e}_{m+k},e_k\right),\bar{D}\right)$ is a model of the fibre $V_m\left(\eta\right)$ and according to Definition~\ref{rt} the bundle \ref{eqmds} is rationally trivial.
 \end{proof}
 We have also the following corollary.
 \begin{coro} 
 \label{coro43}
   If $e\left(\eta\right)=0$, and if  $\; \forall \ell\geq \frac{k}{2}, \; p_\ell\left(\eta\right)=0$, then 
  \[
  V_m\left(\eta\right)\simeq_{\mathbb{Q}}Y\times V_m\left(\mathbb{R}^{m+k}\right).
  \]
  \end{coro}
  \begin{proof}
  From Theorem~\ref{modelfibredestiefel} and on hypotheses on the Pontryagin and Euler classes of $\eta$, the differential $\bar{D}$ of the fibre $V_m\left(\eta\right)$ satisfies 
 \begin{itemize}
 \item
$\bar{D}_{\mid V}=d,$
\item
$\bar{D}\bar{a}_{\frac{k}{2}}= p_{\frac{k}{2}}\left(\eta\right)\otimes 1-1\otimes e_k^2$ if $k$ is even,
\item
 $\bar{D}\bar{a}_{\ell}= p_{\ell}\left(\eta\right)\otimes 1 \text{ for } \ell >\frac{k}{2},$
\item
$ \bar{D}\bar{e}_{m+k}=e\left(\eta\right)\otimes 1,$
\end{itemize}  
 Therefore we have an isomorphism
 \[
 \left( \Lambda V\otimes \Lambda\left(\{\bar{a}_{\ell}\}_{\frac{k}{2}\leq \ell\leq \frac{m+k-1}{2}},\bar{e}_{m+k},e_k\right),\bar{D}\right)\cong
 \left( \Lambda V,d\right)\otimes\left( \Lambda\left(\{\bar{a}_{\ell}\}_{\frac{k}{2}\leq \ell\leq \frac{m+k-1}{2}},\bar{e}_{m+k},e_k\right),d'\right),
 \]
 and the corollary follows.
 \end{proof}
 The rest of this section is devoted to the proof of Theorem~\ref{modelfibredestiefel}.\\
 
  Before to start the computation leading to the proof, let us give the main idea of the proof .  Theorem~\ref{corollaire333} assures that if $\xi$ and $\eta$ are vector bundles respectively classified by $f\colon X\rightarrow G_m\left(\mathbb{R}^\infty\right)$ and $g\colon Y\rightarrow G_{m+k}\left(\mathbb{R}^\infty\right)$, the construction of Stiefel associated to $\xi$ and $\eta$ is obtained by the following double pullbacks
\begin{eqnarray}
\label{abc}
  \xymatrix{
\Stieffel\left(\xi,\eta\right)\ar[dd]_{\pi_1}\ar@/^2pc/[rr] ^{\pi_2} \ar[r]& \Stieffel\left(\gamma^m,\eta\right)\ar[d]\ar@{}[rd]^-{p.b.(2)} \ar[r]\ar[d]&Y\ar[d]^{g}\\
\ar[d]\ar@{}[rd]^{p.b.(1)}&G_m\left(\mathbb{R}^{\infty}\right)\times G_k\left(\mathbb{R}^\infty\right)\ar[r]_-{\rho_{m,k}} \ar[d]^{pr_1}& G_{m+k}\left(\mathbb{R}^\infty\right)\\
X \ar[r]_{f}&G_m\left(\mathbb{R}^\infty\right).
}
 \end{eqnarray}
By \cite[Theorem~2.70]{amt08}, to determine the relative Sullivan model of the Stiefel bundle $\Stieffel\left(\xi,\eta\right)\overset{\pi_1}\rightarrow X$, we only need:
\begin{enumerate}
\item
a cgda model of the diagram
  \begin{eqnarray}
  \label{equoto}
  \xymatrix{
  G_m\left(\mathbb{R}^\infty\right)\times G_k\left(\mathbb{R}^\infty\right)\ar[rr]^-{\rho_{m,k}}\ar[d]_{pr_1}&&G_{m+k}\left(\mathbb{R}^\infty\right)\\
  G_m\left(\mathbb{R}^\infty\right)
  }
  \end{eqnarray}
  \item
  cgda models of $f\colon X\rightarrow G_m\left(\mathbb{R}^\infty\right)$ and $g\colon Y\rightarrow G_{m+k}\left(\mathbb{R}^\infty\right)$.
 \end{enumerate}
 Recall that the cohomology of any oriented Grassmann manifold  $G_n\left(\mathbb{R}^\infty\right)$ is a free polynomial algebra on even degree generator (corresponding to the Pontryagin classes and, when $n$ is even, the top Pontryagin class replace by its square root which is the Euler class). Therefore a cgda model of (\ref{equoto}) is obtained by just taking the cohomology algebra and since $\rho_{m,k}$ is the universal map classifying the Whitney sum, the induced map in cohomology  is given by the classical formula for the Pontryagin classes (and Euler classes) of a Whitney sum. More precisely, let us denote the Pontryagin classes of $G_m\left(\mathbb{R}^\infty\right), G_{m+k}\left(\mathbb{R}^\infty\right)$ and $G_k\left(\mathbb{R}^\infty\right)$ by, respectively, $p_t\in H^{4t}\left(G_{m}\left(\mathbb{R}^\infty\right),\mathbb{Q}\right), a_i\in H^{4i}\left(G_{m+k}\left(\mathbb{R}^\infty\right),\mathbb{Q}\right)$, $b_j\in H^{4j}\left(G_{k}\left(\mathbb{R}^\infty\right),\mathbb{Q}\right)$ and their Euler classes by $e_m,e_{m+k}$ and $e_k$ (which only appears when the corresponding subscript is even). Then the cgda model of (\ref{equoto}) is given by
     \begin{eqnarray*}
  \xymatrix{
\left( \Lambda\left(p_1,\cdots,p_{\lfloor\frac{m-1}{2}\rfloor},e_m,b_1\cdots b_{\lfloor\frac{k-1}{2}\rfloor},e_k\right),0\right)&&\left(\Lambda\left(a_1,\cdots,a_{\lfloor\frac{m+k-1}{2}\rfloor},e_{m+k}\right),0\right)\ar[ll]_-{\tilde{\rho}_{m,k}}\\
 \left( \Lambda\left(p_1,\cdots, p_{\lfloor\frac{m-1}{2}\rfloor},e_m\right),0\right)\ar@{^{(}->}[u]^{\iota}
  }
  \end{eqnarray*} 
with the convention that the generator $e_n$ only appear when is even. The Whitney somme formula  (\cite[Theorem~15.3 and Property~9.6]{Milnor}) imply that

 \[
\tilde{\rho}_{m,k} \left(a_i\right)=\sum_{i=t+j}p_tb_j,  \text{ it is understood that } p_0=b_0=1,
\]
and
\begin{eqnarray*}
\tilde{\rho}_{m,k}\left(e_{m+k}\right)=
\begin{cases}
\begin{aligned}
&e_me_k& \text{   if $m$ and $k$ are even}\\
&0 &\text{   if $m$ and $k$ are odd}.
\end{aligned}
\end{cases}
\end{eqnarray*}
We will replace $\tilde{\rho}_{m,k}$ by a relative Sullivan model, in which will make appear dual Pontryagin classes $\tilde{p}_t$.\\
Let us now build in details this  relative Sullivan model of $\rho_{m,k}$.\\
 The dual Pontryagin classes $\tilde{p}_t\in H^{4t}\left(G_{m}\left(\mathbb{R}^\infty\right),\mathbb{Q}\right)$ are characterized  by the equation
\[
\left(1+p_1+\cdots +p_{\lfloor\frac{m-1}{2}\rfloor}\right)\left(1+\tilde{p}_1+\cdots\right)=1.
\]
We will show that a relative Sullivan algebra 
  \begin{eqnarray}
   \label{adgc3313}
\left(A_1,D_1\right)= \left(\Lambda\left(a_1,\cdots a_{\lfloor\frac{m+k-1}{2}\rfloor},e_{m+k}\right)\otimes \Lambda\left(\bar{a}_{\lceil\frac{k}{2}\rceil},\cdots \bar{a}_{\lfloor\frac{m+k-1}{2}\rfloor}, p_1,\cdots p_{\lfloor\frac{m-1}{2}\rfloor},e_m,e_k,\bar{e}_{m+k}\right),D_1\right)
 \end{eqnarray}
where the differential $D_1$ is defined by:
\begin{itemize}
\item
$D_1a_i=D_1e_{m+k}=D_1e_m=D_1p_t=D_1e_k=0$,
\item
$D_1\bar{a}_{\frac{k}{2}}=a_{\frac{k}{2}}+e_k^2-\tilde{p}_{\frac{k}{2}}$ if $k$ is even\\ 
\item
$D_1\bar{a}_{\ell}=a_{\ell}-\tilde{p}_{\ell}$ if $\ell>\frac{k}{2}$,
\item
$D_1\bar{e}_{m+k}=e_{m+k}-e_me_k$.
\end{itemize}
is a relative model of Diagram~(\ref{equoto}).
 \begin{lem}
 \label{modele6}
Let $\left(A_1,D_1\right)$ be the relative Sullivan algebra constructed in (\ref{adgc3313}). 
Let $\left(\Lambda\left(a_1,\cdots,a_{\lfloor\frac{m+k-1}{2}\rfloor},e_{m+k}\right),0\right)$ and $\left(\Lambda\left(p_1,\cdots,p_{\lfloor\frac{m-1}{2}\rfloor},e_m\right),0\right)$ be the minimal model of  $G_{m+k}\left(\mathbb{R}^\infty\right)$ and $G_m\left(\mathbb{R}^\infty\right)$ respectively. The diagram
 \begin{eqnarray*}
  \newcommand{\immouv}[1][r]
   {\ar@{}[#1] |[F]{\hbox{%
         \vrule width 1.5mm height 0pt depth 0pt%
         \vrule width 0pt height .75mm depth .75mm%
         }}
     \ar@{^{(}->}[#1]}
  \xymatrix{
\left(A_1,D_1\right)&\left(\Lambda\left(a_1\cdots,a_{\lfloor\frac{m+k-1}{2}\rfloor},e_{m+k}\right),0\right) \immouv[l]\\
\left(\Lambda\left(p_1\cdots,p_{\lfloor\frac{m-1}{2}\rfloor},e_m\right),0\right) \ar@{^{(}->}[u],
}
\end{eqnarray*}
is a relative Sullivan model of Diagram~(\ref{equoto}).
 \end{lem}
 For the proof of this result we need the following lemma, in which we use the projection $pr_1\colon G_m\left(\mathbb{R}^\infty\right)\times G_k\left(\mathbb{R}^\infty\right)\rightarrow G_m\left(\mathbb{R}^\infty\right)$ and the map $\rho_{m,k} \colon G_m\left(\mathbb{R}^\infty\right)\times G_k\left(\mathbb{R}^\infty\right)\rightarrow G_{m+k}\left(\mathbb{R}^\infty\right)$ .
  \begin{lem}
 \label{ahl}
In $H^\ast \left(G_m\left(\mathbb{R}^\infty\right)\times G_k\left(\mathbb{R}^\infty\right),\mathbb{Q}\right)$ and for $i,j, t\geq 0$, let $f_i=a_i-\sum_{t+j=i}p_tb_j$, where $a_i,b_j, p_t$ for $i,j,t\geq 1$ are as in Lemma~\ref{modele6}  and $p_0=b_0=1$. Let $\tilde{p}_{i}, i\geq 1$ be the homogeneous polynomial of degree $4i$ satisfying  
  \[
  \left(1+p_1+\cdots p_{\lfloor\frac{m-1}{2}\rfloor}\right)\left(1+\tilde{p}_{1}+\cdots\right)=1,
  \]
then $\forall \ell> \lceil\frac{k}{2}\rceil $ we have the following relations
   \[
  a_{\ell}-\tilde{p}_{\ell}=-\left(\tilde{p}_{\ell-1}\left(a_1-f_1\right)+\cdots+\left(a_{\ell}-f_{\ell}\right)\right),
  \]  
  and if  $k$ is even, we have
   \[
  a_{\frac{k}{2}}-e_k^2-\tilde{p}_{\frac{k}{2}}=-\left(\tilde{p}_{\frac{k}{2}-1}\left(a_1-f_1\right)+\cdots+\left(a_{\frac{k}{2}}-f_{\frac{k}{2}}\right)\right).
  \]    \end{lem}
 \begin{proof}
First of all, note that 
\begin{eqnarray*}
 1+f_1+\cdots+f_{\lfloor\frac{m+k-1}{2}\rfloor}=1+a_1+\cdots a_{\lfloor\frac{m+k-1}{2}\rfloor}-\left(\sum_{t+j=1}p_tb_j+\cdots+\sum_{t+j=\lfloor\frac{m+k-1}{2}\rfloor}p_tb_j\right)\\
=2+a_1+\cdots a_{\lfloor\frac{m+k-1}{2}\rfloor}-\left(1+\sum_{t+j=1}p_tb_j+\cdots+\sum_{t+j=\lfloor\frac{m+k-1}{2}\rfloor}p_tb_j\right)\\
 =2+a_1+\cdots a_{\lfloor\frac{m+k-1}{2}\rfloor}-\left(1+p_1+\cdots+p_{\lfloor\frac{m-1}{2}\rfloor}\right)\left(1+b_1+\cdots b_{\lfloor\frac{k-1}{2}\rfloor}\right)
\end{eqnarray*}
consequently
  \begin{eqnarray*}
\left(1+p_1+\cdots+p_{\lfloor\frac{m-1}{2}\rfloor}\right)\left(1+b_1+\cdots b_{\lfloor\frac{k-1}{2}\rfloor}\right)
=1+\left(a_1-f_1\right)+\cdots+\left(a_{\lfloor\frac{m+k-1}{2}\rfloor}-f_{\lfloor\frac{m+k-1}{2}\rfloor}\right).
\end{eqnarray*}
Since
 \[
 \left(1+\tilde{p}_{1}+\cdots\right)= \left(1+p_1+\cdots+p_{\lfloor\frac{m-1}{2}\rfloor}\right)^{-1}
 \]
 on one hand we have $\forall i\geq 1, \tilde{p}_{i}$ is a polynomial of $p_1,\cdots p_{\lfloor\frac{m-1}{2}\rfloor}$ and on the other hand
 \[
 1+b_1+\cdots b_{\lfloor\frac{k-1}{2}\rfloor}=\left(1+\left(a_1-f_1\right)+\cdots+\left(a_{\lfloor\frac{m+k-1}{2}\rfloor}-f_{\lfloor\frac{m+k-1}{2}\rfloor}\right) \right)\left(1+{p}_1+\cdots\right)^{-1}
  \]
In particular if  $\ell > \frac{k}{2}$ we obtain
  \[
  a_{\ell}-{}\tilde{p}_{\ell}=-\left(\tilde{p}_{\ell-1}\left(a_1-f_1\right)+\cdots+\left(a_{\ell}-f_{\ell}\right)\right),
  \]
  if $k$ is even $b_{\frac{k}{2}}=e_k^2$ henceforth  
  \[
  a_{\ell}+e_k^2-\tilde{p}_{\frac{k}{2}}=-\left(\tilde{p}_{\frac{k}{2}-1}\left(a_1-f_1\right)+\cdots+\left(a_{\frac{k}{2}}-f_{\frac{k}{2}}\right)\right).
  \]

 \end{proof}
 \begin{proof}[Proof of Lemma~\ref{modele6}]  
 The rational cohomology of Diagram~(\ref{equoto}) is given by the diagram
     \begin{eqnarray*}
  \xymatrix{
 \mathbb{Q}\left[p_1,\cdots,p_{\lfloor\frac{m-1}{2}\rfloor},e_m,b_1,\cdots b_{\lfloor\frac{k-1}{2}\rfloor},e_k\right]&&\mathbb{Q}\left[a_1,\cdots,a_{\lfloor\frac{m+k-1}{2}\rfloor},e_{m+k}\right]\ar[ll]_-{\tilde{\rho}_{m,k}}\\
  \mathbb{Q}\left[p_1,\cdots, p_{\lfloor\frac{m-1}{2}\rfloor},e_m\right]\ar@{^{(}->}[u]^{\iota}
  }
  \end{eqnarray*} 
 where $\iota$ is the canonical inclusion, and by \cite[Theorem~15.3 and Property~9.6]{Milnor}
 \[
\tilde{\rho}_{m,k}\left(a_i\right)=\sum_{i=t+j}p_tb_j \text{  and   }\; \tilde{\rho}_{m,k}\left(e_{m+k}\right)=e_me_k.
 \]
 Thus
      \begin{eqnarray*}
  \xymatrix{
\left( \Lambda\left(p_1,\cdots,p_{\lfloor\frac{m-1}{2}\rfloor},e_m,b_1\cdots b_{\lfloor\frac{k-1}{2}\rfloor},e_k\right),0\right)&&\left(\Lambda\left(a_1,\cdots,a_{\lfloor\frac{m+k-1}{2}\rfloor},e_{m+k}\right),0\right)\ar[ll]_-{\tilde{\rho}_{m,k}}\\
 \left( \Lambda\left(p_1,\cdots, p_{\lfloor\frac{m-1}{2}\rfloor},e_m\right),0\right)\ar@{^{(}->}[u]^{\iota}
  }
  \end{eqnarray*} 
 
  is a model. This model is not a relative Sullivan model. In order to obtain a relative Sullivan model
we first transform $\tilde{\rho}_{m,k}$ to a relative model, such as
  \begin{eqnarray*}
  \xymatrix{
\left(A_1',D_1'\right) &&\left(\Lambda\left(a_1,\cdots,a_{\lfloor\frac{m+k-1}{2}\rfloor},e_{m+k}\right),0\right)\ar@{^{(}->}[ll]\\
  \left(\Lambda\left(p_1,\cdots,p_{\lfloor\frac{m-1}{2}\rfloor},e_m\right),0\right)\ar@{^{(}->}[u]^{\iota}}
 \end{eqnarray*} 
  with
 \[
  \left(A_1',D_1'\right)=\left(\Lambda\left(p_1,\cdots, p_{\lfloor\frac{m-1}{2}\rfloor},e_m,b_1,\cdots, b_{\lfloor\frac{k-1}{2}\rfloor},e_k\right)\otimes \Lambda\left(x_1,\cdots, x_{\lfloor\frac{m+k-1}{2}\rfloor}, \bar{e}_{m+k},a_1,\cdots, a_{\lfloor\frac{m+k-1}{2}\rfloor},{e}_{m+k}\right),D_1'\right)
  \]
  where
  \begin{itemize}
  \item 
 $x_i$ are generators of degree $4i-1$,
  \item
$D_1'a_i=D_1'e_{m+k}=
D_1'b_j=D_1'p_t=D_1'e_m=D_1'e_k=0$,
\item
$D_1\bar{a_i}= a_i- \tilde{\rho}_{m,k}\left(a_i\right) =a_i-\sum_{t+j=i}p_tb_j$,
\item
$D_1\bar{e}_{m+k}=e_{m+k}- \tilde{\rho}_{m,k}\left(e_{m+k}\right)=e_{m+k}-e_me_k$.
\end{itemize} 
To end we have to prove that this cgda is isomorphic to $\left(A_1,D_1\right)$. For this we define a map

\[
\phi\colon \left(A_1,D_1\right)\longrightarrow \left(A_1',D_1'\right)
\]  by
\begin{itemize}
\item
$\phi\left(a_i\right)=a_i, \phi\left(e_{m+k}\right)=e_{m+k}, \phi\left(p_t\right)=p_t, \phi\left(e_k\right)=e_k,\phi\left(e_m\right)=e_m$,
\item
$\phi\left(\bar{a}_{\ell}\right)=-\left(\tilde{p}_{\ell-1} x_1+\tilde{p}_{\ell-2} x_2+\cdots {x}_{\ell}\right)$
\item
$\phi\left(\bar{e}_{m+k}\right)=\bar{e}_{m+k}$.
\end{itemize}
This morphism is a cgda morphism. In fact, we have
\begin{itemize}
\item
$\phi\left(D_1a_i\right)=D'_1a_i=0, \phi\left(D_1e_{m+k}\right)=D'_1e_{m+k}=0, \phi\left(D'_1p_t\right)=D'_1p_t=0, \phi\left(D_1e_k\right)=D_1'e_k=0,$\\
$\phi\left(D_1e_m\right)=D_1'e_m=0$,
\item
$\phi\left(D_1\bar{a}_{\ell}\right)=\phi\left(a_{\ell}-\tilde{p}_{\ell}\right)=a_{\ell}-\tilde{p}_{\ell} $ and
\begin{eqnarray*}
D_1'\left(\tilde{p}_{\ell-1}\bar{a}_1+\tilde{p}_{\ell-2}\bar{a}_2+\cdots \bar{a}_{\ell}\right)&=&\left(\tilde{p}_{\ell-1} D_1'\bar{a}_1+\tilde{p}_{\ell-2} D'_1\bar{a}_2+\cdots D'_1 \bar{a}_{\ell}\right)\\
&=&\tilde{p}_{\ell-1}\left(a_1-(p_1+b_1)\right)+\cdots+{p}_1(a_{\ell}-\sum_{t+j=\ell} p_tb_j)\\
&=&-\left(a_\ell-{p}_\ell\right) \text{ (by Lemma~\ref{ahl})}.
\end{eqnarray*}
Then $\forall \ell >\frac{k}{2}$ we have $\phi\left(D_1\bar{a}_{\ell}\right)= D_1'\phi\left(\bar{a}_{\ell}\right)$. Similarly we can prove that 
\[
\phi\left(D_1{x}_{\frac{k}{2}}\right)= a_{\frac{k}{2}}-e_k^2-\tilde{p}_{\frac{k}{2}}=D_1'\phi\left(\bar{a}_{\frac{k}{2}}\right).
\]
\item
$\phi\left(D_1\bar{e}_{m+k}\right)=\phi\left(e_{m+k}-e_me_k\right)=e_{m+k}-e_me_k=D_1'\phi\left(\bar{e}_{m+k}\right)$.
\end{itemize}
This proves that  $\phi$ is a cgda  morphism and it is clear that it induces an isomorphism on cohomology.
 \end{proof}

Now we give the proof of Theorem~\ref{modelfibredestiefel}
  \begin{proof}[Proof of Theorem~\ref{modelfibredestiefel}]
    By \cite[Theorem~2.70]{amt08} , applied to the homotopy pullback (2) of  Diagram~(\ref{abc}), the model of $\Stieffel\left(\gamma^m,\eta\right)$ is of the form
 \begin{eqnarray*}
  \left(\Lambda{V},d\right)\otimes_{\left(\Lambda\left(\{a_i\},e_{m+k}\right),0\right)}\left(A_1,D_1\right)
 \cong \left(\Lambda{V}\otimes\Lambda\left(\{\bar{a}_\ell\},\{p_t\},e_m,e_k, ,\bar{e}_{m+k}\right),D_2\right)
  \end{eqnarray*}
  \begin{itemize}
\item
${D_2}_{\mid\Lambda V}=d$,
\item
 $D_2\bar{a}_{\frac{k}{2}}=p_{\frac{k}{2}}\left(\eta\right)\otimes{1}-1\otimes\tilde{p}_{\frac{k}{2}}-1\otimes e_k^2 $\text{ if }$k$ is even,
\item
$D_2\bar{a}_\ell=p_\ell\left(\eta\right)\otimes{1}-1\otimes\tilde{p}_\ell \text{ if } \ell>\frac{k}{2}$ 
\item
$D_2\bar{e}_{m+k}=e\left(\eta\right)\otimes 1-e_me_k$,
\item
$D_2p_t=D_2e_m=D_2e_k=0$.
 \end{itemize} 
 
By \cite[Theorem~2.70]{amt08}, applied to the homotopy pullback (1) of Diagram~(\ref{abc}), the model of 
$\Stieffel\left(\xi,\eta\right)$ is of the form
\begin{eqnarray*}\left(A,d_A\right)\otimes_{\left(\Lambda\left(\{p_t\},e_m\right),0\right)}\left(\Lambda{V}\otimes\Lambda\left(\{\bar{a}_i\},\{p_t,\},e_m,e_k,\bar{e}_{m+k}\right),D_2\right)
\cong\left(A\otimes\Lambda{V}\otimes\Lambda\left(\{\bar{a}_i\},e_k,\bar{e}_{m+k}\right),D\right)
\end{eqnarray*}
where
    \begin{itemize}
\item
$D_{\mid A}=d_A,\; {D}_{\mid\Lambda V}=d$,
\item
$D\bar{a}_\ell=1\otimes p_\ell\left(\eta\right)\otimes{1}-\tilde{p}_\ell\left(\xi\right)\otimes 1\otimes 1$ if $\ell>\frac{k}{2}$  and $D\bar{a}_{\frac{k}{2}}=1\otimes p_\ell\left(\eta\right)\otimes{1}-e_k^2-\tilde{p}_{\frac{k}{2}}\left(\xi\right)\otimes 1\otimes 1$,
\item
$D\bar{e}_{m+k}=1\otimes e\left(\eta\right)\otimes 1-e\left(\xi\right)\otimes 1\otimes e_k$,
\item
$De_k=0$.
 \end{itemize} 
\end{proof}

   \section{Rational homotopy of the space of immersions}
 \label{modelimm}
 In this section we study the rational homotopy type of the space of immersions of a manifold $M$ in a manifold $N$. 
\subsection{Smale-Hirsch Theorem}
We start  with a quick review of  the Smale-Hirsch theorem.
By applying the construction of Stiefel to the tangent bundles $\tau_M,$ and $\tau_N$ we obtain  the Stiefel bundle
\[
 V_m\left(\tau_N\right)\longrightarrow \Stieffel\left(\tau_M,\tau_N\right) \overset{\pi_1}\longrightarrow M.
\]
Let $\Gamma\left( \Stieffel\left(\tau_M,\tau_N\right) \overset{\pi_1}\longrightarrow M\right)$ denote the space of sections of  the Stiefel bundle. That is 
\[
\Gamma\left( \Stieffel\left(\tau_M,\tau_N\right)\longrightarrow M\right)=\{ s\colon M\longrightarrow \Stieffel\left(\tau_M,\tau_N\right) | \; \pi_1\circ s= id\}
\]
with the compact-open topology. The following result relates the space $\Gamma\left( \Stieffel\left(\tau_M,\tau_N\right)\rightarrow M\right)$ with the space of immersions of $M$ into $N$. It is the fundamental theorem  of immersion theory.
\begin{thm}\cite[Smale-Hirsch theorem]{hirsh}
\label{sht}
 Let $M$ be a manifold of dimension $m$ and let $N$ be a manifold of dimension $m+k$ with $k\geq 1$ an integer. There is an homotopy equivalence
\[
s\colon Imm\left(M,N\right)\longrightarrow \Gamma\left( \Stieffel\left(\tau_M,\tau_N\right)\longrightarrow M\right).
\]
\end{thm}
\begin{rmq}
In general  the spaces $Imm\left(M,N\right)$ and  $\Gamma\left( \Stieffel\left(\tau_M,\tau_N\right)\rightarrow M\right)$ are not connected. For each section $s_0\colon M\rightarrow \Stieffel\left(\tau_M,\tau_N\right) $ we denote by  $\; \Gamma_{s_0}\left( \Stieffel\left(\tau_M,\tau_N\right)\rightarrow M\right)$ the space of sections of\; $\Stieffel\left(\tau_M,\tau_N\right)$ homotopic to $s_0$. That is, 
\[
\Gamma_{s_0}\left( \Stieffel\left(\tau_M,\tau_N\right)\longrightarrow M\right)=\{s\colon M\longrightarrow \Stieffel\left(\tau_M,\tau_N\right) | \; \pi_1\circ s= id \text{ and } s\simeq s_0\}.\]

 By the Smale-Hirsch theorem, for each immersion $f\colon M\looparrowright N$ there exists a section
 \[
 s\left(f\right)\colon M\longrightarrow \Stieffel\left(\tau_M,\tau_N\right)
 \]
  such that $Imm\left(M,\mathbb{R}^{m+k};f\right)$ and $\Gamma_{s\left(f\right)}\left( \Stieffel\left(\tau_M,\tau_N\right)\rightarrow M\right)$ are homotopically equivalent.
 \end{rmq}
 In the following we will use this result to study the rational homotopy type of the space of immersions of a manifold in a manifold. 
  \subsection{Rational homotopy of the space of immersions of a manifold in a manifold}
  In this section we prove the following theorem, which identify the component of the space of immersions of a manifold $M$ in a manifold $N$ with some mapping space.
 \begin{thm}
 \label{theorem1}
 Let $M$ and $N$ be two manifolds of dimension $m$ and $m+k$ respectively, simply connected and of finite type, and $f\colon M \looparrowright N$ an immersion. Suppose that all the Pontryagin classes of $N$ are zeros. If  one of the following hypothesis satisfy:
 \begin{itemize}
 \item
  $k$ is odd, or
 \item
 $e\left(\tau_M\right)=0$ and  $\tilde{p}_{\frac{k}{2}}\left(\tau_M\right)=0$, 
 \end{itemize}
 then  we have the following rational homotopy equivalence
  \begin{eqnarray*}
 \label{eq:img}
 Imm\left(M,N;f\right)\simeq_{\mathbb{Q}}Map\left(M,V_m\left(\tau_N\right);\phi_f\right)
 \end{eqnarray*}
 where $\phi_f$ is some continuous  map from $M$ to $V_m\left(\tau_N\right)$.
\end{thm}
For the proof of this theorem we need the following lemma.
\begin{lem}
\label{federico}Let $f$ be an  immersion from $M$ to $N$ and $\nu_f$ the normal bundle of $f$.
If all the rational Pontryagin classes of $N$ are zero, then  we have the following equalities:
\begin{eqnarray*}
\begin{cases}
\begin{aligned}
&\tilde{p}_i\left(\tau_M\right)={p}_i\left(\nu_f\right)&& \text{ for } i\leq\lfloor\frac{k-1}{2}\rfloor&\\
&\tilde{p}_i\left(\tau_M\right)=0 &&\text{ for }i>\lfloor\frac{k-1}{2}\rfloor&
\end{aligned}
\end{cases}
\end{eqnarray*}
\end{lem}
\begin{proof}
 By \cite[Corollary~3.5]{Milnor}, we have
\[
f^\ast\tau_N\cong \tau_M\oplus \nu_f,
\]
then by \cite[Proposition~3.5]{ad} we have the following equality 
\[
1+p_1\left(f^\ast\tau_N\right)+\cdots p_{\lfloor{\frac{m+k-1}{2}}\rfloor}\left(f^\ast\tau_N\right)=1+p_1\left( \tau_M\oplus \nu_f \right)+\cdots+p_{\lfloor\frac{m+k-1}{2}\rfloor}\left( \tau_M\oplus \nu_f \right) .
\]
 On the other hand by \cite[Theorem~15.3]{Milnor}
\[
\forall i\geq 1 \text{ we have }\;  p_i\left( \tau_M\oplus \nu_f\right)=\sum_{t+j=i}p_t\left(\tau_M\right)p_j\left(\nu_f\right).
\]
where $t,j$ are positives integers and $p_0\left(\tau_M\right)=p_0\left(\nu_f\right)=1$.\\
Thus
\begin{eqnarray*}
1+p_1\left(f^\ast\tau_N\right)+\cdots p_{\lfloor{\frac{m+k-1}{2}}\rfloor}\left(f^\ast\tau_N\right)
&=&1+\sum_{t+j=1}p_t\left(\tau_M\right)p_j\left(\nu_f\right)+\cdots +\sum_{t+j=\lfloor\frac{m+k-1}{2}\rfloor}p_t\left(\tau_M\right)p_j\left(\nu_f\right)\\
&=&
\left(1+p_1\left(\tau_M\right)+\cdots+p_{\lfloor{\frac{m-1}{2}}\rfloor}\left(\tau_M\right)\right)\left(1+{p}_1\left(\nu_f\right)+\cdots +p_{\lfloor{\frac{k-1}{2}}\rfloor}\left(\nu_f\right)\right).
\end{eqnarray*}
Since $\forall i\geq 1, p_i\left(\tau_N\right)=0$ then $\forall i\geq 1, p_i\left(f^\ast\tau_N\right)=0$ and 
\begin{eqnarray*}
&&\left(1+p_1\left(\tau_M\right)+\cdots+p_{\lfloor{\frac{m-1}{2}}\rfloor}\left(\tau_M\right)\right)\left(1+{p}_1\left(\nu_f\right)+\cdots +p_{\lfloor{\frac{k-1}{2}}\rfloor}\left(\nu_f\right)\right)=1.
\end{eqnarray*}
By unicity of dual Pontryagin classes, we have
\[
1+{p}_1\left(\nu_f\right)+\cdots+p_{\lfloor{\frac{k-1}{2}}\rfloor}\left(\nu_f\right)=1+\tilde{p}_1\left(\tau_M\right)+\cdots 
\]
i.e
\begin{eqnarray*}
\begin{cases}
\begin{aligned}
&\tilde{p}_i\left(\tau_M\right)={p}_i\left(\nu_f\right) &&\text{ for } i\leq\lfloor\frac{k-1}{2}\rfloor&\\
&\tilde{p}_i\left(\tau_M\right)=0&& \text{ for }i>\lfloor\frac{k-1}{2}\rfloor&
\end{aligned}
\end{cases}
\end{eqnarray*}
\end{proof}
\begin{proof}[Proof of Theorem~\ref{theorem1}]
Let $f\colon M\looparrowright N$ be an immersion and $Imm\left(M,N;f\right)$ the component of $Imm\left(M,N\right)$ containing $f$. By the Smale-Hirsch theorem we have
\[
Imm\left(M,N;f\right)\simeq \Gamma_{s\left(f\right)}\left( \Stieffel\left(\tau_M,\tau_N\right)\longrightarrow M\right),
\]
and by Lemma~\ref{federico} for all $i>\frac{k}{2}, \tilde{p}_i\left(\tau_M\right)=0$. Since $k$ is odd or  $e\left(\tau_M\right)=0$ and $\tilde{p}_{\frac{k}{2}}\left(\tau_M\right)=0$, by Corollary ~\ref{corodecoro} the Stiefel bundle associated to $\tau_M$ and $\tau_N$
\[
V_m\left(\tau_N\right)\longrightarrow  \Stieffel\left(\tau_M,\tau_N\right)\overset{\pi_1}\longrightarrow M
\]
is rationally trivial. Therefore, its rationalization
\[
V_m\left(\tau_N\right)_{\mathbb{Q}}\longrightarrow \left( \Stieffel\left(\tau_M,\tau_N\right)\right)_{\mathbb{Q}}\longrightarrow M_{\mathbb{Q}}
\] 
is trivial and by Corollary~\ref{corodecoro1}, its  fiberwise rationalization
\[
V_m\left(\tau_N\right)_{\mathbb{Q}}\longrightarrow \left( \Stieffel\left(\tau_M,\tau_N\right)\right)_{\left(\mathbb{Q}\right)}\longrightarrow M
\] 
is trivial.
Since $k>1$ and $N$ simply connected,  then  $V_m\left(\tau_N\right)$ is simply-connected.  Considering that $M$ is of finite type, we have from \cite{moller}
\[
\Gamma_{s\left(f\right)}\left( \Stieffel\left(\tau_M,\tau_N\right)\longrightarrow M\right)\simeq \Gamma_{rs\left(f\right)}
\left(\left( \Stieffel\left(\tau_M,\tau_N\right)\right)_{\left(\mathbb{Q}\right)}\longrightarrow M\right)
\]
where 
\[
r\colon  \Stieffel\left(\tau_M,\tau_N\right)\longrightarrow \left( \Stieffel\left(\tau_M,\tau_N\right)\right)_{\left(\mathbb{Q}\right)}
\] 
is the fiberwise rationalization. Since 
$
\left( \Stieffel\left(\tau_M,\tau_N\right)\right)_{\left(\mathbb{Q}\right)}\longrightarrow M
$
is trivial. We have 
\[
\Gamma_{rs\left(f\right)}
\left(\left( \Stieffel\left(\tau_M,\tau_N\right)\right)_{\left(\mathbb{Q}\right)}\longrightarrow M\right)\simeq Map\left(M,V_m\left(\tau_N\right)_{\mathbb{Q}};\phi_f\right)
\]
where $\phi_f \colon M\longrightarrow V_m\left(\tau_N\right)_{\mathbb{Q}}$ is the component of $rs\left(f\right)$
\end{proof}

\begin{coro} 
\label{coro1}
Under the hypothesis of Theorem~\ref{theorem1}, suppose in addition that the Euler class of $N$ is zero. Then there exists two maps 
 $\phi_f^1\colon M\rightarrow N$ and  $\phi_f^2\colon M\longrightarrow V_m\left(\mathbb{R}^{m+k}\right)$  such that
     \[
   Imm\left(M,N;f\right)\simeq_{\mathbb{Q}}Map\left(M,N,\phi_f^1\right)\times Map\left(M,V_m\left(\mathbb{R}^{m+k}\right);\phi_f^2\right).
   \]
\end{coro}

\begin{proof}
In fact, if the Euler class of $N$ is zero and  for all $i\geq 1\; p_i\left(\tau_N\right)=0$ by Corollary~\ref{coro43} we have 
 \[
    V_m\left(\tau_N\right)\simeq_{\mathbb{Q}} N\times V_m\left(\mathbb{R}^{m+k}\right).
   \]
    Therefore by Theorem~\ref{theorem1}  we deduce the following
      \begin{eqnarray*}
      Imm\left(M,N;f\right)&\simeq_{\mathbb{Q}}&Map\left(M,N\times V_m\left(\mathbb{R}^{m+k}\right);\phi_f\right)\\
      &\simeq_{\mathbb{Q}}&Map\left(M,N;\phi_f^1\right)\times Map\left(M, V_m\left(\mathbb{R}^{m+k}\right);\phi_f^2\right) \qedhere     
       \end{eqnarray*}
\end{proof}
 \subsection{Rational homotopy of the space of immersions of manifold in an Euclidean space}
 In this section we study the rational homotopy type of the space of immersions of manifolds in an Euclidean space. According to the codimension, we will distinguish two cases.
\subsubsection{Odd codimension case}
\begin{thm}
\label{casodd}
Let $M$ be a manifold of dimension $m\geq 0$, simply connected and of finite type. If $k\geq 3$ is an odd integer, then each connected component of $Imm\left(M,\mathbb{R}^{m+k}\right)$ has the rational homotopy type of a product of Eilenberg-Mac Lane spaces which only depends  on the dimension $m$, the codimension $k$ and the rational Betti numbers of $M$.
\end{thm}
Moreover, we have an explicit description of each component of $Imm\left(M,\mathbb{R}^{m+k}\right)$.  Precisely we have.
\begin{prop}
\label{corod}
Let  $f \colon M\looparrowright \mathbb{R}^{m+k}$an immersion, and let  $Imm\left(M ,\mathbb{R}^{m+k};f\right)$ be the connected component of $Imm\left(M ,\mathbb{R}^{m+k}\right)$ containing $f$. Then, we have the following rational homotopy equivalence:
\[
Imm\left(M ,\mathbb{R}^{m+k};f\right)\simeq_{\mathbb{Q}}
\begin{cases}
\begin{aligned}
&\prod_{\lceil\frac{k}{2}\rceil\leq i\leq \lfloor\frac{m+k}{2}\rfloor}\prod_{1\leq q\leq 4i-1}K\left(H^{4i-1-q}\left(M,\mathbb{Q}\right),q\right)& &\text{ if } &m \text{ even}\\
&\prod_{\lceil\frac{k}{2}\rceil\leq i\leq \lfloor\frac{m+k-1}{2}\rfloor}\prod_{1\leq q\leq 4i-1}K\left(H^{4i-1-q}\left(M,\mathbb{Q}\right),q\right)\times {A}&&\text{ if }& m \text{ odd}
\end{aligned}
\end{cases}
\]
where 
\begin{eqnarray*}
A={\prod_{1\leq j\leq m+k-1}K\left(H^{m+k-1-j}\left(M,\mathbb{Q}\right),j\right)}.
\end{eqnarray*}
\end{prop}
\begin{proof}[Proof of Theorem~\ref{casodd} and Proposition~\ref{corod}]
The Euler class and all the Pontryagin classes of $ \;\mathbb{R}^{m+k}$ are zero. Since $M$ is simply connected, of finite type and $k\geq 3$, by Corollary~\ref{coro1} we have
   \begin{eqnarray}
  \label{12}
  Imm\left(M,\mathbb{R}^{m+k};f\right)\simeq_{\mathbb{Q}}Map\left(M,\mathbb{R}^{m+k};\phi_f^1\right)\times Map\left(M,\V_m\left(\mathbb{R}^{m+k}\right);\phi_f^2\right).
 \end{eqnarray}
Since $\mathbb{R}^{m+k}$ is contractible, the space $Map\left(M,\mathbb{R}^{m+k};\phi_f^1\right)$ is contractible. Therefore, the rational homotopy equivalence (\ref{12}) becomes  
   \begin{eqnarray*}
  \label{13}
  Imm\left(M,\mathbb{R}^{m+k};f\right)\simeq_{\mathbb{Q}} Map\left(M,V_m\left(\mathbb{R}^{m+k}\right);\phi_f^2\right).
 \end{eqnarray*}
 Thus we need to understand the rational homotopy of such a component of a mapping space of $M$ into the Stiefel manifold $V_m\left(\mathbb{R}^{m+k}\right)$.
Since $k$ is odd by \cite[Proposition~3.1]{pierre77} the Stiefel manifold  $V_m\left(\mathbb{R}^{m+k}\right)$ has the rational homotopy of a product of Eilenberg-Mac Lane space depending on $m$ and $k$. Moreover any component of space of $M$ into such a product of Eilenberg-Mac Lane spaces is itself a product of Eilenberg-Mac Lane spaces which only depends on $m$ and $k$ and on the Betti numbers of $M$. This establishes Theorem~\ref{casodd}.\\
More precisely we have  
\[
V_m\left(\mathbb{R}^{m+k}\right)\simeq_{\mathbb{Q}}
\begin{cases}
\begin{aligned}
&\prod_{\lceil\frac{k}{2}\rceil\leq\ell\leq \lfloor\frac{m+k-1}{2}\rfloor}K\left(\mathbb{Q},4\ell-1\right)\times K\left(\mathbb{Q},m+k-1\right)&&\text{ if $m$ is odd}&\\
&\prod_{\lceil\frac{k}{2}\rceil\leq\ell\leq \lfloor\frac{m+k}{2}\rfloor}K\left(\mathbb{Q},4\ell-1\right) &&\text{ if $m$ is even}&.
\end{aligned}
\end{cases}
\]
Consequently 
\begin{eqnarray*}
&& Map\left(M,V_m\left(\mathbb{R}^{m+k}\right)_{\mathbb{Q}};\phi_f\right)\\
&\simeq&
\begin{cases}
\begin{aligned}
 &Map\left(M, \prod_{\lceil\frac{k}{2}\rceil\leq \ell\leq \lfloor\frac{m+k-1}{2}\rfloor}K\left(\mathbb{Q},4\ell-1\right)\times K\left(\mathbb{Q},m+k-1\right);\phi_f\right) &&\text{ if $m$ is odd}&\\
&Map\left(M, \prod_{\lceil\frac{k}{2}\rceil\leq \ell\leq \lfloor\frac{m+k}{2}\rfloor}K\left(\mathbb{Q},4\ell-1\right);\phi_f\right)&&\text{ if $m$ is even}&
\end{aligned}
\end{cases}
\\
&\simeq&
\begin{cases}
\begin{aligned}
&\prod_{\lceil\frac{k}{2}\rceil\leq \ell\leq \lfloor\frac{m+k-1}{2}\rfloor} Map\left(M,K\left(\mathbb{Q},4\ell-1\right);\phi_f^{\ell}\right)\times Map\left(M,K\left(\mathbb{Q},m+k-1\right);\phi_f^{m+k-1}\right)&&\text{ if $m$ is odd}&\\
&\prod_{\lceil\frac{k}{2}\rceil\leq \ell\leq \lfloor\frac{m+k}{2}\rfloor} Map\left(M,K\left(\mathbb{Q},4\ell-1\right);\phi_f^{\ell}\right)&&\text{if $m$ is even}&
\end{aligned}
\end{cases}
 \end{eqnarray*}
 By \cite[Theorem~1]{moller} we have 
 \[
 Map\left(M,K\left(\mathbb{Q},4\ell-1\right);\phi_f^{\ell}\right)\simeq_{\mathbb{Q}}\prod_{1\leq q\leq 4\ell-1}K\left(H^{4\ell-1-q}\left(M,\mathbb{Q}\right),q\right),
  \]
  and
  \[
 Map\left(M,K\left(\mathbb{Q},m+k-1\right);\phi_f^{m+k-1}\right)\simeq_{\mathbb{Q}}\prod_{1\leq q\leq m+k-1}K\left(H^{m+k-1-q}\left(M,\mathbb{Q}\right),q\right).
  \]
  This ends the proof.
\end{proof}

\subsubsection{Even codimension case}
In all this section $k\geq 2$ is an even integer. This case is slightly different from the odd codimensional case because when $k$ is even, the Stiefel manifold $V_m\left(\mathbb{R}^{m+k}\right)$ has the rational homotopy of a product of Eilenberg-Mac Lane space with an even dimensional sphere $S^k$. Because of this we need some extra hypothesis on $M$ to get an easy description the rational homotopy of component of $Imm\left(M,\mathbb{R}^{m+k}\right)$. We will prove the following result.
\begin{thm}
\label{caseven}
Let $M$ be a manifold of dimension $m\geq0$, simply connected and of finite type. Assume that the following two statement hold
\begin{itemize}
\item $e\left(\tau_M\right)=0,$ and
\item
$\tilde{p}_{\frac{k}{2}}\left(\tau_M\right)=0$ then all the component of $Imm\left(M,\mathbb{R}^{m+k}\right)$.
\end{itemize}
Then each connected component of $Imm\left(M,\mathbb{R}^{m+k}\right)$ has the rational homotopy type of a product of  some component of $Map\left(M,S^k\right)$ with a product of Eilenberg-Mac Lane spaces which only depends on $m$, $k$ and the rationals Betti numbers of $M$.\\
If moreover $H^k\left(M,\mathbb{Q}\right)=0$ then all the components of $Imm\left(M,\mathbb{R}^{m+k}\right)$ have the same rational homotopy type as well as all the components of $Map\left(M,S^k\right)$.
\end{thm}
Moreover, we have the following description of each component of $Imm\left(M,\mathbb{R}^{m+k}\right)$.
\begin{prop}
\label{coreven}
Let $f \colon M\looparrowright \mathbb{R}^{m+k}$ be an immersion, and let  $Imm\left(M ,\mathbb{R}^{m+k},f\right)$ be the connected component of $Imm\left(M ,\mathbb{R}^{m+k}\right)$ containing $f$. Under the hypothesis of  (i) and (ii)  of Theorem~\ref{caseven}, we have the following rational homotopy equivalence:
\[
Imm\left(M ,\mathbb{R}^{m+k};f\right)\simeq_{\mathbb{Q}}
\begin{cases}
\begin{aligned}
&\prod_{i=\frac{k}{2}+1}^{\lfloor\frac{m+k-1}{2}\rfloor}\prod_{1\leq q\leq 4i-1}K\left(H^{4i-1-q}\left(M,\mathbb{Q}\right),q\right )\times {Map\left(M,S^k;\tilde{f}\right)}&&\text{ if } &m \text{ odd}\\
&\prod_{i=\frac{k}{2}+1}^{\lfloor\frac{m+k-1}{2}\rfloor}\prod_{1\leq q\leq 4i-1}K\left(H^{4i-1-q}\left(M,\mathbb{Q}\right),q\right) \times{Map\left(M,S^k;\tilde{f}\right) }\times {A}&&\text{ if }& m \text{ even},
\end{aligned}
\end{cases}
\]
where
$A={\prod_{1\leq j\leq m+k-1}K\left(H^{m+k-1-j}\left(M,\mathbb{Q}\right),j\right)}$ and $\tilde{f}\colon M\rightarrow S^k$ is some map depending on the immersion $f$.
\end{prop}
\begin{rmq} Remark that when $k$ is even in general  the rational homotopy type of $ \; Map\left(M,S^k,\tilde{f}\right)$ depends on the homotopy class of $\tilde{f}
$ (see \cite{Moller1} for example). Therefore, when the codimension is even, the rational homotopy type depends on the choice of immersion. 
\end{rmq}
For the proof of the second part of Theorem~\ref{caseven} we need the following lemma.
\begin{lem}
\label{yvess}
Let $M$ be a manifold of dimension $m$ and $k\geq0$ an even integer.  If $H^k\left(M, \mathbb{Q}\right)=0$, then all the components of $Map\left(M,S^k\right)$ have the same rational homotopy type.
\end{lem}
\begin{proof}
A model of the component of mapping space $Map\left(M,S^k;\psi\right)$ is given by the Haefliger model \cite{HA82} of the trivial bundle 
\[
M\times S^k\longrightarrow M
\]
associated to the section 
\[
s\colon M\longrightarrow M\times S^k, x\mapsto \left(x,\psi\left(x\right)\right).
\]
Let $\left(A,d_A\right)$ a finite dimensional cgda model of $M$ and $\left(\Lambda\left(x,y\right),dy=x^2\right)$ be the minimal model of $S^k$ with $\mid{x}\mid=k, \mid{y}\mid=2k-1$. Let
\[
\sigma\colon A\otimes \Lambda\left(x,y\right)\longrightarrow A
\]
be a model of $s$ such that $\sigma_{\mid A}=id_A$.\\
The Haefliger model of $Map\left(M,S^k;\psi\right)$ can explicitly  be constructed from the cdga $\left(A,d_A\right)\otimes\left(\Lambda{\left(x,y\right)},d \right)$ and the cgda map $\sigma$.\\
Note that $\sigma\left(x\right)$ is a cocycle of degree $k$ is $A$ and  since $H^k(A,d)=0$, we can   homotop $\sigma$  to some map which sends $x$ to $0$. Moreover $\sigma\left(y\right)$ is a cocycle $a\in A$ and replay the generator  $y$ by $y-a$, we can also assume that $\sigma\left(y\right)=0$. Thus we can assume that $\sigma_{\mid {\Lambda\left(x,y\right)}}=0$ which would be the model of the section associated to the map $\psi\colon M\rightarrow S^k$. Therefore 
$ Map\left(M,S^k;\psi\right)$ and $ Map\left(M,S^k;\ast\right)$ have the same cgda model and hence 
 the same rational homotopy type. 
\end{proof}

\begin{proof}[Proof of Theorem~\ref{caseven} and Proposition~\ref{coreven}]
Since $e\left(\tau_M\right)=0$, $\tilde{p}_{\frac{k}{2}}\left(\tau_M\right)=0$, $\forall i\geq 1, p_i\left(\tau_{\mathbb{R}^{m+k}}\right)=0$ and $e\left(\tau_{\mathbb{R}^{m+k}}\right)=0$ by Corollary~\ref{coro1} we have the following rational homotopy equivalence
 \[
 Imm\left(M,\mathbb{R}^{m+k};f\right)\simeq_{\mathbb{Q}} Map\left(M,\mathbb{R}^{m+k};\phi_f^1\right)\times Map\left(M,V_m\left(\mathbb{R}^{m+k}\right);\phi_f^2\right)\simeq_{\mathbb{Q}}Map\left(M,V_m\left(\mathbb{R}^{m+k}\right);\phi_f^2\right) \]
 and  by Proposition~\cite[Proposition~3.1]{pierre77} we have
\[
V_m\left(\mathbb{R}^{m+k}\right)\simeq_{\mathbb{Q}}
\begin{cases}
\prod_{\frac{k}{2}+1\leq\ell\leq \lfloor\frac{m+k-1}{2}\rfloor}K\left(\mathbb{Q},4\ell-1\right)\times K\left(\mathbb{Q},m+k-1\right)\times S^k \text{ if $m$ is even}\\

\prod_{\frac{k}{2}+1\leq\ell\leq \lfloor\frac{m+k-1}{2}\rfloor}K\left(\mathbb{Q},4\ell-1\right)\times S^k \text{ if $m$ is odd}
\end{cases}
\]
The first part of the proof is analogue to the proof of Propoition~\ref{corod}.\\
Assume now the extra hypothesis $H^k\left(M,\mathbb{Q}\right)=0$. Then the theorem follow directly from Lemma~\ref{yvess}
\end{proof}

From the results we deduce the following corollary.
\begin{coro} 
\label{betti}
Let $M$ be a simply connected manifold of dimension $m$ and $k\geq m$ an integer. If $k$ is odd or $e\left(\tau_M\right)=0$, then the rational Betti numbers of  $Imm\left(M,\mathbb{R}^{m+k}\right)$ have polynomial growth.
\end{coro}
\begin{rmq}
\label{alp}
In \cite{alp} Arone, Lambrechts and Pryor prove that, if $\chi\left(M\right)\leq -2$ and $k\geq {m+1}$ the rational Betti numbers of space of smooth embeddings of $M$ in $\mathbb{R}^{m+k} $ modulo immersions, $\overline{Emb}\left(M ,\mathbb{R}^{m+k}\right)$, have exponential growth. By Corollary~\ref{betti} we deduce that the Betti numbers of smooth embedding $Emb\left(M ,\mathbb{R}^{m+k}\right)$ have exponential growth.
\end{rmq}

\section{Series of the ranks of homotopy groups the connected components of the space of immersions}
\label{computation}
Proposition~\ref{corod} and Proposition~\ref{coreven} show that the rational homotopy type of the space of immersions of a manifold  in an euclidian space with large codimension depends only on the codimension and the rational Betti numbers of the manifold. From this we can obtain an explicit formula for the ranks of  the rational homotopy groups of the component of space of immersions. More precisely, when the codimension is odd we obtain the following result.
\begin{thm}
\label{seriepoincare4}
Let  $M$  be a simply connected manifold closed  of dimension $m$  and $k\geq 3$ an odd integer. If $k\geq \frac{m}{2}+1$ , then:
 \[
 \sum_{i=1}^{\infty}rank\;\pi_i\left( Imm\left(M,\mathbb{R}^{m+k};f\right)\right)x^i
 =
 \begin{cases}
 \begin{aligned}
& \frac{x^{2k-m-1}\left(1-x^{2m+2}\right)}{1-x^4}P_M\left(x\right) &&\text{ if $m$ even}&\\
 &\frac{x^{3k-m-2}\left(1-x^{2m+2}\right)}{1-x^4}\left(P_M\left(x\right)\right)^2& &\text{ if $m$ odd}&
 \end{aligned} 
 \end{cases} 
 \]
 where
 \[
 P_M\left(x\right)=\sum \dim H_i\left(M,\mathbb{Q}\right)x^i \text{ is a  Poincar\'e serie of } M.
 \]
\end{thm}
\begin{proof}
 By Proposition~\ref{corod} if  $f\colon M\looparrowright \mathbb{R}^{m+k}$ is an  immersion we have
 \[
Imm\left(M ,\mathbb{R}^{m+k};f\right)\simeq_{\mathbb{Q}}
\begin{cases}
\begin{aligned}
&\prod_{\lceil\frac{k}{2}\rceil\leq i\leq \lfloor\frac{m+k-1}{2}\rfloor}\prod_{1\leq q\leq 4i-1}K\left(H^{4i-1-q}\left(M,\mathbb{Q}\right),q\right) &&\text{ if }& m \text{ is even}\\
&\prod_{\lceil\frac{k}{2}\rceil\leq i\leq\lfloor \frac{m+k-1}{2}\rfloor}\prod_{1\leq q\leq 4i-1}K\left(H^{4i-1-q}\left(M,\mathbb{Q}\right),q\right)\times {K}&&\text{ if } &m \text{ is odd} ,
\end{aligned}
\end{cases}
\]
where
\begin{eqnarray*}
A={\prod_{1\leq j\leq m+k-1}K\left(H^{m+k-1-j}\left(M,\mathbb{Q}\right),j\right)}.
\end{eqnarray*}
Consequently 
\begin{itemize}
\item{ if $m$ is even} we have
\begin{eqnarray*}
 \sum_{i=1}^{\infty}rank\;\pi_i\left( Imm\left(M,\mathbb{R}^{m+k};f\right)\right)x^i =\sum_{i=\frac{k}{2}}^{\frac{m+k-1}{2}}\sum_{q=1}^{4i-1}b_{4i-1-q}\left(M\right)x^q.
 \end{eqnarray*}
Pose $q'=4i-1-q$, then we obtain 
 \begin{eqnarray*}
 \sum_{i=1}^{\infty}rank\;\pi_i\left( Imm\left(M,\mathbb{R}^{m+k},f\right)\right)x^i &=&\sum_{i=\frac{k}{2}}^{\frac{m+k-1}{2}}\sum_{q'=4i-2}^{0}b_{q'}\left(M\right)x^{4i-1-q'}\\
 &=&\sum_{i=\frac{k}{2}}^{\frac{m+k-1}{2}}x^{4i-1}\sum_{q'=4i-2}^{0}b_{q'}\left(M\right)x^{-q'}
  \end{eqnarray*}
 Pose $q"=m-q'$, then
  \begin{eqnarray*}
  \sum_{i=1}^{\infty}rank\;\pi_i\left( Imm\left(M,\mathbb{R}^{m+k};f\right)\right)x^i =\sum_{i=\frac{k}{2}}^{\frac{m+k-1}{2}}x^{4i-1-m}\sum_{q''=m-4i+2}^{m}b_{m-q''}\left(M\right)x^{q''}.
  \end{eqnarray*}
By Poincar\'e duality  we have 
\[
b_{m-i}\left(M\right)=b_i\left(M\right)
\]
consequently
  \begin{eqnarray*}
   \sum_{i=1}^{\infty}rank\;\pi_i\left( Imm\left(M,\mathbb{R}^{m+k};f\right)\right)x^i =\sum_{i=\frac{k}{2}}^{\frac{m+k-1}{2}}x^{4i-1-m}\sum_{q''=m-4i+2}^{m}b_{q''}\left(M\right)x^{q''}.
  \end{eqnarray*}
  On the other hand , since $\forall i<0, b_i\left(M\right)=0$ then
  \[
  \forall i\geq \frac{m+2}{4}, \text{ we have } \sum_{q''=m-4i+2}^{m}b_{q''}\left(M\right)x^{q''}=\sum_{q''=0}^{m}b_{q''}\left(M\right)x^{q''}.
    \]
 Since $k\geq \frac{m}{2}+1$ (by hypothese ) we have:
  \begin{eqnarray*}
  \sum_{i=\frac{k}{2}}^{\frac{m+k-1}{2}}x^{4i-1-m}\sum_{q''=m-4i+2}^{m}b_{q''}\left(M\right)x^{q''}=\sum_{i=\frac{k}{2}}^{\frac{m+k-1}{2}}x^{4i-1-m}\sum_{q''=0}^{m}b_{q''}\left(M\right)x^{q''}.
    \end{eqnarray*}
Then
    \begin{eqnarray*}
     \sum_{i=1}^{\infty}rank\;\pi_i\left( Imm\left(M,\mathbb{R}^{m+k};f\right)\right)x^i={\left(1+x^4+\cdots+x^{2m-2}\right)}{x^{2k-m-1}}Betti_M\left(x\right).
     \end{eqnarray*}
   To finish, remark that
    \begin{eqnarray*}
    1+x^4+\cdots x^{2m-2}=\frac{1-x^{2m+2}}{1-x^4}.
    \end{eqnarray*}
    
\item{ If $m$ is odd}, we have
  \begin{eqnarray*}
  \sum_{i=1}^{\infty}rank\;\pi_i\left( Imm\left(M,\mathbb{R}^{m+k};f\right)\right)x^i =\sum_{i=\frac{k}{2}}^{\frac{m+k-1}{2}}x^{4i-1-m}\sum_{q''=0}^{m}b_{m-q''}\left(M\right)x^{q''}\times{K},
 \end{eqnarray*}
 where 
 \begin{eqnarray*}
 K=x^{k-1}\sum_{j=-k+2}^{m}b_j\left(M\right)x^j=x^{k-1}\sum_{j=0}^{m}b_j\left(M\right)x^j ;\text{ since }k\geq 3.
 \end{eqnarray*}
  \end{itemize} 
  The  rest  of the proof is the same as in the case  $m$ even.
\end{proof}
When the codimension is  even  we have the following result.
\begin{thm}
\label{seriepoincare2}
Let $M$  be a simply connected manifold closed  of dimension $m$  and $k\geq 2$ an even integer. Suppose that $k\geq\frac{m}{2}+1$
. Then, if $H^k\left(M,\mathbb{Q}\right)=0$ and $e\left(\tau_M\right)=0$ then
\begin{eqnarray*}
 \sum_{i=1}^{\infty}rank\; \pi_i\left( Imm\left(M,\mathbb{R}^{m+k};f\right)\right)x^i
  =
  \begin{cases}
  \begin{aligned}
  &\frac{x^{2k-m-1}\left(1-x^{2m}\right)}{1-x^4}P_M\left(x\right)\times R\left(x\right)&& \text{ if $m$ is odd}&\\
  &\frac{x^{3k-m-2}\left(1-x^{2m}\right)}{1-x^4}\left(P_M\left(x\right)\right)^2\times R\left(x\right)&& \text{ if $m$ is even},&
  \end{aligned}
    \end{cases}
\end{eqnarray*}
where
\[
P_M\left(x\right)= \sum \dim H_i\left(M,\mathbb{Q}\right)x^i  \text{ and } R\left(x\right)=\sum_{i}\left(\dim H^{k-i}\left(M,\mathbb{Q}\right)-\dim H^{2k-i-1}\left(M,\mathbb{Q}\right)\right)x^i.
\]
\end{thm}

\bibliographystyle{plain}
\bibliography{mabiblio}

\textsf{Universit\'e catholique de Louvain, Chemin du Cyclotron 2, B-1348 Louvain-la-Neuve, Belgique\\
Institut de Recherche en Math\'ematique et Physique\\}
\textit{E-mail address: abdoul.yacouba@uclouvain.be}
\end{document}